\documentclass[11pt,reqno]{amsart}
\usepackage[utf8]{inputenc}
\usepackage[T1]{fontenc}

\usepackage[american]{babel}
\usepackage{amsmath}
\usepackage{dsfont,mathtools,amssymb}
\usepackage{tikz}
\usepackage{subcaption} % Replace subfigure with subcaption
\usepackage{color}
\usepackage{mathrsfs}

\usepackage[colorlinks=true, linkcolor=blue, citecolor=blue, urlcolor=blue]{hyperref}%
\mathtoolsset{showonlyrefs,showmanualtags}

\definecolor{darkblue}{rgb}{0.1,0.1,0.7}

\definecolor{darkred}{rgb}{0.7,0.1,0.1}

\newtheorem{theorem}{Theorem}[section]

\newtheorem{lemma}[theorem]{Lemma}
\newtheorem{corollary}[theorem]{Corollary}
\newtheorem{remark}[theorem]{Remark}
\newtheorem{definition}[theorem]{Definition}

\numberwithin{equation}{section}
\numberwithin{figure}{section}

\newcommand{\scalar}[2]{\langle #1 , #2\rangle}

\newcommand{\ind}{\mathbf{1}}

\newcommand{\e}{\varepsilon}

\renewcommand{\rho}{\varrho}
\renewcommand{\phi}{\varphi}

\newcommand{\IND}{{\bf 1}}

\DeclareMathOperator{\var}{Var}
\DeclareMathOperator{\cov}{Cov}

\DeclareMathOperator{\Ent}{Ent}

\newcommand{\grad}{\nabla}

\newcommand{\be}{\begin{equation}} 
\newcommand{\Om}{\Omega}

\newcommand{\cC}{\ensuremath{\mathcal C}}

\newcommand{\cF}{\ensuremath{\mathcal F}}

\newcommand{\cP}{\ensuremath{\mathcal P}}

\newcommand{\cS}{\ensuremath{\mathcal S}} 
\newcommand{\cT}{\ensuremath{\mathcal T}}

\newcommand{\cW}{\ensuremath{\mathcal W}}

\newcommand{\bbE}{{\ensuremath{\mathbb E}} }

\newcommand{\bbN}{{\ensuremath{\mathbb N}} }

\newcommand{\bbR}{{\ensuremath{\mathbb R}} }

\newcommand{\si}{\sigma} 
\newcommand{\ent}{{\rm Ent} } 
\let\a=\alpha \let\b=\beta   \let\d=\delta  \let\e=\varepsilon
 \let\g=\gamma     \let\k=\kappa  \let\l=\lambda
      \let\o=\omega      
\let\r=\rho  \let\s=\sigma \let\t=\tau   
  
\let\D=\Delta   \let\G=\Gamma   
\let\O=\Omega

\def\({\left(}
\def\){\right)}
\begin{document}
	
	\begin{abstract}
	We derive entropy factorization estimates for spin systems using the stochastic localization approach proposed by Eldan \cite{eld13} and Chen-Eldan \cite{cheneldan}, which, in this context, is equivalent to the renormalization group approach developed independently by Bauerschmidt, Bodineau, and Dagallier \cite{BBD2023}.
%		We derive entropy factorization estimates for spin systems  using the stochastic localization approach proposed by Eldan \cite{eld13} and Chen-Eldan \cite{cheneldan}, which in this context is equivalent to the renormalization group approach developed independently by Bauerschmidt, Bodienau and Dagallier \cite{BBD2023}. 
The method provides approximate Shearer-type inequalities for the corresponding Gibbs measure at sufficiently high temperature, without restrictions on the degree of the underlying graph. For Ising systems, these are shown to hold up to the critical tree-uniqueness threshold, including polynomial bounds at the critical point, with optimal $O(\sqrt n)$ constants for the Curie-Weiss model at criticality. In turn, these estimates imply tight mixing time bounds for arbitrary block dynamics or Gibbs samplers,  
improving over existing results.	Moreover, we establish new tensorization statements for the Shearer inequality asserting that if a system consists of weakly interacting but otherwise arbitrary components, each of which satisfies an approximate Shearer inequality, then the whole system also satisfies such an estimate.  
	\end{abstract}

\author{Pietro Caputo}
\address{Department of Mathematics and Physics, Roma Tre University.}
\email{pietro.caputo@uniroma3.it}
\author{Zongchen Chen}
\address{School of Computer Science, Georgia Institute of Technology.}
\email{chenzongchen@gatech.edu}
\author{Daniel Parisi}
\address{Department of Mathematics, University Roma La Sapienza.}
\email{daniel.parisi@uniroma1.it}

	\title[entropy factorizations using stochastic localization]{Factorizations of Relative Entropy \\ Using Stochastic Localization} %[entropy factorizations through stochastic localization]{A note on entropy factorizations \\ through stochastic localization}
	\maketitle

\section{Introduction}
\thispagestyle{empty}
Consider a spin system on a finite graph $G=(V,E)$ such as the Potts model with $q$ colors. This is the probability measure $\mu$ on $\O=[q]^V$ given by 
\begin{align}\label{spinsys0}
	\mu(\si) = \frac1{Z}\,\exp\(2\beta\textstyle{\sum_{ij\in E}}\IND_{\sigma_i=\sigma_j} + \textstyle{\sum_{i\in V}} \psi_i(\sigma_i)\) \,,
\end{align}
where $\si=(\si_i,\,i\in V)$ is the spin configuration, $Z$ is the normalization constant, $\psi_i$ are given self-potentials, and $\b\in\bbR$ is a parameter measuring the strength of the interaction.  
The model is ferromagnetic if $\b\ge 0$ and antiferromagnetic otherwise. When $q=2$ this is the Ising model on the graph $G$ with external fields.
The Glauber dynamics for the spin system is the Markov chain where at each step one vertex is chosen uniformly at random and its spin value is updated with a sample from the conditional distribution given all other spin values.   
When the graph has bounded maximum degree, optimal $O(n\log n)$ mixing time bounds for the Glauber dynamics of such spin systems were obtained recently in \cite{CLV21} under the assumption of spectral independence. Here $n=|V|$ is the size of the vertex set, and the bounds are optimal in the sense that they catch, up to a constant factor, the correct growth rate of the mixing time as a function of $n$. 

The notion of spectral independence, introduced in \cite{ALO}, provides a powerful tool for the analysis of weakly dependent spin systems. For instance, spectral independence is known to hold % it allows one to cover
for the ferromagnetic Ising model, for all values of $\b$ in the tree-uniqueness region, uniformly in the choice of the external fields \cite{CLV21,FGYZ2022}. The estimates in \cite{CLV21} were then extended in \cite{BCCPSV22} to obtain optimal mixing time bounds for arbitrary block dynamics, where at each step a whole block of variables rather than a single vertex is updated with a sample from the conditional distribution given the remaining variables. Both papers \cite{CLV21,BCCPSV22} used extensively the method of entropy factorizations, a systematic study of which was introduced in \cite{CP2021}. Entropy factorization inequalities are naturally linked to Gibbs samplers or heat-bath block dynamics. They are known to imply lower bounds on the relative entropy contraction rate for arbitrary Gibbs samplers, which in turn provide mixing time bounds for the corresponding  Markov chain. They can also be interpreted as approximate versions of the Shearer inequality, a well-known family of inequalities satisfied by the relative entropy with respect to a product measure. 
Combined with spectral independence, the method of entropy factorizations turned out to be rather powerful in the context of bounded-degree graphs, where it also allows one to obtain optimal $O(\log n)$ mixing time bounds for the Swendsen-Wang dynamics \cite{BCCPSV22}. 
An extension to unbounded-degree graphs remained however a challenging task. We refer to \cite{AnariEntInd,CFYZ} for optimal results for the Glauber dynamics for two-valued spin systems under spectral independence without restrictions on the degrees. 
Extending these results to arbitrary block dynamics was the main motivation for our present work. 

On the other hand, Bauerschmidt, Bodineau, and Dagallier \cite{BB19,BD2022,BBD2023} proposed a different approach based on renormalization group ideas, which allows one to obtain tight entropy contraction estimates for the Glauber dynamics of Ising-like systems under a high-temperature assumption expressed in terms of a covariance estimate, without restrictions on the degrees. An essentially equivalent method, based on the stochastic localization process introduced earlier by Eldan \cite{eld13}, was then proposed by Chen and Eldan \cite{cheneldan}. The latter work provides a unified framework allowing a simple derivation of several recently established facts, including a bound for the Sherrington--Kirkpatrick model that was previously obtained in \cite{EKZ2022,AJKPV2021}, as well as the results for Glauber dynamics given in  \cite{AnariEntInd,CFYZ}. Very recently, this approach was extended considerably in \cite{CCYZ25}, to handle two-spin valued systems up to and including the tree-uniqueness threshold, with polynomial bounds at the uniqueness threshold; see also \cite{prodromidis2024polynomial}.

Our first contribution in this paper is an extension of the results from \cite{BD2022,cheneldan} and \cite{CCYZ25} to cover arbitrary block dynamics, under the same high-temperature assumptions. Moreover, we formulate our results in the general setting of arbitrary $q$-valued  spin systems rather than just two-valued spins.   In the special case of the graphical Ising model, this yields tight entropy factorization estimates in the whole tree-uniqueness regime, including polynomial bounds at the uniqueness threshold, with no restriction on the degrees, thereby improving the results established in \cite{BCCPSV22}. An interesting question that remains open is to obtain $O(\log n)$ mixing time bounds for the Swendsen-Wang dynamics of the ferromagnetic Potts model, thus extending one of the results of \cite{BCCPSV22} to the unbounded-degree setting.

Motivated by applications to mean field particle system approximations to nonlinear dynamics for spin systems \cite{CapSin}, we also consider an extension of these results to spin systems consisting of several weakly interacting components, with minimal assumptions on the distribution of the single components. One of the results we obtain in this case is that if the interaction between the components is suitably weak, then the components can be approximately factorized. Furthermore, we establish a new tensorization statement for the approximate Shearer inequality asserting that if a system consists of weakly interacting but otherwise arbitrary components, each of which 
satisfies an approximate Shearer inequality, then the whole system also satisfies such an estimate.   

Our proofs closely follow the stochastic localization approach of \cite{cheneldan}, with the main differences being that we address arbitrary Gibbs samplers instead of focusing solely on Glauber dynamics, consider spins with 
$q$ values rather than restricting to 2-valued spins, and explicitly handle spin systems with multiple interacting components. For the reader's convenience, we provide an essentially self-contained treatment.

\section{Setting and main results}\label{sect:mainres}
\subsection{Spin system}
A spin system is a probability measure over spin configurations. A spin configuration is a map $\si:V\mapsto [q]$, where $V$ is a finite set of vertices, and $[q]=\{1,\dots,q\}$ is the set of spin values. Here $q\in\bbN$, $q\ge 2$ is fixed. The set $V$ is identified with $[n]=\{1,\dots,n\}$, where $n\in\bbN$ is the number of vertices.  We write  $\Omega=[q]^V$ for the set of spin configurations and $\cP(\O)$ for the set of probability measures on $\O$.   We shall consider spin systems of the form $\mu\in \cP(\O)$, with 
\begin{align}\label{spinsys}
	\mu(\si) = \frac1{Z}\,\exp\(\sum_{i,j=1}^n\phi_{i,j}(\sigma_i,\sigma_j) + \sum_{i=1}^n \psi_i(\sigma_i)\) \,,
\end{align}
where $Z$ is the normalization constant. 
For each $i,j\in [n]$,  $\phi_{i,j}:[q]\times[q]\mapsto\bbR$ is a  given function representing the {\em pair interaction}, and satisfying the symmetry $ \phi_{i,j}(a,b)=\phi_{j,i}(b,a)$, while $\psi_i:[q]\mapsto\bbR$ are given functions called the {\em self potentials}. % or {\em external fields}. 
We refer to \eqref{spinsys} as the Gibbs measure. 

The main examples we have in mind are spin systems such as the {\em Ising model} and the {\em Potts model} on a graph. If $G=(V,E)$ is a finite graph, the $q$-Potts model on $G$ with parameter $\beta\in\bbR$ and zero external field is obtained by taking 
\begin{align}\label{eq:Potts}
\phi_{i,j}(a,b) = \beta A(G)_{i,j}\IND_{a=b}\,,
\end{align} 
for all $a,b\in[q]$, where $  A(G)$ denotes the adjacency matrix of $G$. When $q=2$ this is the Ising model on $G$. We refer to \cite{FriedliVelenik} for the classical statistical mechanics background %on the  statistical mechanical properties of 
concerning such spin systems. For simplicity, here we do not discuss systems with hard constraints such as hard-core gases or graph colorings. 
We shall however consider a class of generalized spin systems later in the paper, which includes conservative systems obtained by imposing certain hard constraints.

\subsection{Entropy factorizations}
We start with some notation. For a function $f:\Omega \mapsto \bbR$, and a probability $\mu\in\cP(\O)$, we write $
\mu(f) = \sum_{\s \in \Omega}  \mu(\s) f(\s)$ for the mean of $f$ with respect to $\mu$. For $f\ge 0$, we use the notation 
\[
\ent
f = \mu(f \log f) - \mu(f)  \log \mu(f)\,,
\]
for the entropy of $f$ with respect to $\mu$, so that %When $f\geq 0$ is such that  $\nu[f] =1$, then
$\ent(f) =\mu(f) H(\rho\,|\, \mu),$ where $\rho\in\cP(\O)$ is the probability $\rho=(f/\mu(f))\mu$ and $H(\rho\,|\, \mu)$ denotes the relative entropy, or Kullback-Leibler divergence, of the distribution $\rho$ with respect to  $\mu$.
Note that $\ent f$ depends implicitly on the reference measure $\mu$. If we need to emphasize this dependence we use the symbol $\ent_\mu f$ instead. However, since $\mu$ will be fixed once for all this should be no source of confusion.

Given $\tau \in \Om$, for any $A \subset V$, we define $\mu_A^\t\in\cP(\O)$ as the conditional distribution $\mu_A^\t(\sigma) \propto \mu(\sigma)\IND_{\si_{A^c}=\t_{A^c}}$, where 
$\sigma_{A^c} =(\si_i,i\in A^c)$, and $A^c=V\setminus A$. 
For $f:\O\mapsto\bbR$ we write $\mu_A f$ for the {\em conditional expectation}, namely the function 
$\mu_A f : \O\mapsto \bbR$ defined by $\mu_A f (\t) := \mu_A^{\t} (f)$. For $f\ge 0$, we define the map $\ent_A f: \O\mapsto\bbR_+$ as 
\begin{align}\label{eq:def-ent-la}
	\ent_A f: \;\t\mapsto 	\ent_{\mu_A^\t} f = \mu_A^\tau\(f  \log f\) - \mu_A^\tau\(f\)  \log\mu_A^\tau\( f\)\,.
\end{align}
With this notation, we have \[
\ent_A f=\mu_A\left(f\log(f/\mu_A(f))\right)\,,\qquad \mu(\ent_A f) =  \mu\left(f\log(f/\mu_A(f))\right)\,.
\]
The following identity for any $f\ge 0$ and $A\subset V$ is an immediate consequence: 
\begin{align}\label{eq:deco}
\ent
f =	\mu(\ent_A f) + \ent
(\mu_A f) \,.
\end{align}

Next, we define the notion of entropy factorization as in \cite{CP2021}. Let $\cW$ denote the set of probability measures over subsets of $V$, that is $\a\in\cW$ is a vector of weights $\a=(\a_A, \,A\subset V)$ with $\a_A\ge 0$ and $\sum_{A\subset V}\a_A=1$. 
For any $\a\in\cW$ let $\gamma(\alpha)=\min_i\sum_{A\ni i}\alpha_A$ denote the minimum covering probability. 
\begin{definition}\label{def:ent_fact}
For any $\a\in\cW$, we say that $\mu$ satisfies the $\a$-{\em entropy factorization} with constant $C_\alpha$ if, for all $f\ge 0$, 
\begin{align}\label{sheaineq}
	\gamma(\alpha) \Ent f \leq C_\alpha \sum_{A \subset V}\alpha_A\,\mu\(\ent_A f\)\,,
\end{align}
where $\gamma(\alpha)=\min_i\sum_{A\ni i}\alpha_A$. If $C_\a\le C$ for all  $\a\in\cW$ we say that the {\em approximate Shearer inequality} holds for $\mu$ with constant $C$. If the latter holds uniformly in the choice of the self potentials $\psi_i$ we say that the {\em strong approximate Shearer inequality} holds for $\mu$ with constant $C$.
\end{definition}
The above definition is motivated by the well-known fact that when $\mu$ is a product measure $\mu=\otimes_{i=1}^n \mu_i$, for arbitrary probabilities $\mu_i$ on $[q]$, then 
 $\mu$ satisfies the $\a$-{\em entropy factorization} with constant $C_\alpha=1$ for every $\a\in\cW$, and that this property of product measures is equivalent to the Shearer inequality satisfied by the Shannon entropy, see e.g.\  \cite{PCnotes}. See also \cite{BaBo,MadimanTetali} for background on the Shearer inequality and its applications. Establishing Shearer-type inequalities for non-product measures is a challenging task that was addressed in several recent works, see \cite{CP2021,BC2024,BCCPSV22,caputo_salez_2024entropy}. Our main results here give sufficient conditions on the pair interactions $\phi_{i,j}$ for the Gibbs measure $\mu$ in \eqref{spinsys} to satisfy the strong approximate Shearer inequality.

\subsection{Block dynamics and entropy contractions}
For $\a\in\cW$, the $\a$-weighted block dynamics, or $\a$-Gibbs sampler, for the Gibbs measure $\mu$, is the Markov chain with state space $\O$ and transition matrix 
\begin{align}\label{eq:pa}
P_{\alpha}(\sigma,\eta)=\sum_{A \subset V}\alpha_A\mu_A^\sigma(\eta)\,.
\end{align}
In words, at each step a block $A$ of vertices is picked with probability $\a_A$, and the spins in the ``block'' $A$ are resampled according to the conditional distribution on $A$ given the spins outside of $A$. 
For a function $f$ we then have 
\[
P_{\alpha}f(\sigma)=\sum_{\eta\in\O}P_{\alpha}(\sigma,\eta)f(\eta)=\sum_{A \subset V}\alpha_A\mu_A^\sigma(f)\,,
\]
or equivalently $P_{\alpha}=\sum_{A \subset V}\alpha_A\mu_A$ which expresses the operator $P_\a$ as a mixture of orthogonal projections in $L^2(\O,\mu)$. In particular, $P_\a$ is self-adjoint, and the Markov chain is reversible with invariant measure $\mu$.  The Markov chain is irreducible iff $\g(\a)>0$, where $\gamma(\alpha)=\min_i\sum_{A\ni i}\alpha_A$. 

The case $\a_A=\frac1n\IND_{|A|=1}$ is usually called the {\em Glauber dynamics}; see e.g.\ \cite{Martinelli99}. Another commonly studied problem is the {\em even/odd Gibbs sampler} for spin systems on a bipartite graph, which is obtained when $\a$ assigns probability $1/2$ to the set of even sites and probability $1/2$ to the set of odd sites, see \cite{CP2021,BCCPSV22} for several examples. %We note that %also refer to for  applications of 
See also  \cite{Hinton} for applications of the even/odd Gibbs sampler to learning with restricted Boltzmann machines.  

The entropy contraction rate is defined as the largest $\d_\a\ge 0$ such that  
\begin{align}\label{ent_contr}
	\ent(P_\alpha f) \leq (1-\delta_\alpha)\ent(f)\,,
\end{align}
for all $f\ge 0$.  By convexity, $\ent(P_\alpha f) \le \sum_{A}\alpha_A \ent(\mu_A f)$, and using \eqref{eq:deco} we see that if $\mu$ satisfies the $\a$-entropy factorization with constant $C_\a$ as in Definition \ref{def:ent_fact}, then %for all $\a\in\cW$, 
\begin{align}\label{fact_contr}
	\delta_\alpha\ge \bar \d_\a:=\g(\a)/C_\a\,.
\end{align}
The inequality \eqref{ent_contr} has consequences on the speed of convergence of the $\a$-Gibbs sampler to the stationary measure $\mu$. A standard linearization argument shows that   \eqref{ent_contr} implies the same inequality with the entropy functional replaced by the variance functional, so that $\d_\a$ is a lower bound on the spectral gap of the Markov chain with transition $P_\a$. 
Moreover, it is known that $\d_\a$ is a lower bound on the modified log-Sobolev constant of the $\a$-Gibbs sampler. Therefore,  the mixing time of the chain can be estimated by using the entropy contraction rate: for $\a\in\cW$, the $\a$-entropy factorization estimate \eqref{sheaineq} with constant $C_\a$  provides an upper bound on the mixing time of the $\a$-Gibbs sampler of the form
\begin{align}\label{fact_contr_mix}
T_{\rm mix}(P_\a)= \frac{C_\a}{\g(\a)}\times O(\log n)\,,
\end{align}
 see e.g.\  \cite{PCnotes} for a proof of these facts.
 For instance, when $\a_A=\frac1n\IND_{|A|=1}$,  if  one has $C_\a\le C$ in \eqref{sheaineq} with $C$ independent of $n$, then the Glauber dynamics has optimal mixing time $O(n\log n)$. Similarly, if a spin system on a bipartite graph has $C_\a\le C$ for the weights $\a$ corresponding to the even/odd Gibbs sampler, then this Markov chain has mixing time $O(\log n)$. We refer to \cite{BCCPSV22} for a discussion of several examples showing that the approximate Shearer inequality from Definition \ref{def:ent_fact} is often an efficient way of establishing optimal mixing for arbitrary block dynamics.  

The entropy contraction rate  $\d_\a$ is equivalently described as the best constant in the inequality
\begin{align}\label{ent_contr2}
	H(\nu P_\a\,|\,\mu)\le (1-\delta_\alpha)H(\nu \,|\,\mu)\,,
\end{align}
for all $\nu\in\cP(\O)$, where $\nu P_\a$ is the distribution after one step, when the initial state is distributed according to $\nu$. As recently noted in \cite{CCGP}, the $\a$-entropy factorization \eqref{sheaineq} is instead equivalent to the {\em half-step} entropy contraction described as follows. Let $\cS=2^V$ denote the set of all subsets of $V$
and consider the probability kernel $K_\a:\O\mapsto\cS\times\Omega$ defined by
\[
K_\a(\sigma,(A,\eta)) = \alpha_A\mu_A^\si(\eta)\,. 
\]
Let $K_\a^*$ denote the adjoint of $K_\a$ defined as 
\[
K_\a^*((A,\eta),\si) = \frac{\mu(\si)K_\a(\sigma,(A,\eta))}{\mu K_\a(A,\eta)}\,, 
\]
where $\mu K_\a\in \cP(\cS\times\Omega)$ is the distribution $\mu K_\a(A,\eta) =\sum_\si \mu(\si)K_\a(\sigma,(A,\eta))$. Then it is not difficult to show that $\mu K_\a(A,\eta) = \alpha_A\mu(\eta)$,  $K_\a^*((A,\eta),\sigma) = \mu_A^\eta(\sigma)$, for all $(A,\eta)\in\cS\times\Omega$, $\sigma\in \Omega$, and that  $P_\a = K_\a K^*_\a$. It follows that if $\nu = f\mu$, then 
\begin{align}\label{ent_contr21}
	H(\nu K_\a\,|\,\mu K_\a)= \sum_{A \subset V}\alpha_A\,\ent (\mu_A f)\,.
\end{align}
In particular, using \eqref{eq:deco}, we see that the $\a$-entropy factorization with constant $C_\a$ as in Definition \ref{def:ent_fact} is equivalent to the statement
\begin{align}\label{ent_contr22}
	H(\nu K_\a\,|\,\mu K_\a)\le (1-\bar\d_\a)H(\nu \,|\,\mu )\,,\qquad \bar \d_\a=\g(\a)/C_\a\,,
\end{align}
where the constant $\bar\d_\a$ is the same as in \eqref{fact_contr}.
This is a quantitative form of the data processing inequality. Since $P_\a = K_\a K^*_\a$, we refer to \eqref{ent_contr22} as a half-step entropy contraction for the $\a$-Gibbs sampler.
We refer to \cite{CCGP} for several other examples of half-step contractions and for comparisons between half-step and one-step contractions. 
\subsection{Main results}
Let $\mu$ be a Gibbs measure as in \eqref{spinsys}. Define the $qn\times qn$ matrix
\begin{align}\label{def:Gamma}
\G(i,a;j,b) = \phi_{i,j}(a,b)\,,\qquad i,j\in[n],\,a,b\in[q]\,,
\end{align}
and write $\l(\G)$ for the maximum eigenvalue of the symmetric matrix $\G$. %$\l(\G)$ for the operator norm of $\G$. 
 Our first result is the following criterium for the validity of the strong approximate Shearer inequality. 
%\begin{theorem}\label{mainthm2} 
%	Suppose that $\Gamma$ is positive semidefinite and that $\l(\G)<1$.	
%	Then for all $\alpha\in\cW$, for any choice of the self potentials $\psi_i:[q]\mapsto \bbR$ in \eqref{spinsys}, % the approximate Shearer inequality
%	\begin{align}\label{sheaineq2}
%		\gamma(\alpha) \,\ent f \leq \left(1-\l(\G)\right)^{-1}%\frac1{1-\l(\G)}
%		\sum_{A \subset V}\alpha_A\,\mu\left(\ent_Af\right)\,,
%	\end{align}
%for all functions $f\ge 0$. That is, $\mu$ satisfies the strong approximate Shearer inequality with constant $C:=\left(1-\l(\G)\right)^{-1}$.
%\end{theorem}
\begin{theorem}\label{mainthm2} 
	Suppose that $\Gamma$ is positive semidefinite and that $\l(\G)\le 1-\delta$ where $\delta \in [0,1]$.	
	For all $\alpha\in\cW$, for any choice of the self potentials $\psi_i:[q]\mapsto \bbR$ in \eqref{spinsys}, for all functions $f\ge 0$, % the approximate Shearer inequality
	\begin{align}\label{sheaineq2}
		\gamma(\alpha) \,\ent f \leq C
		\sum_{A \subset V}\alpha_A\,\mu\left(\ent_Af\right)\,,
	\end{align}
	 where
	\begin{align*}
		C := 
		\begin{cases}
			ne^{1-\delta n}, & \text{if } \delta \in [0,\frac{1}{n}); \\
			\frac{1}{\delta}, & \text{if } \delta \in [\frac{1}{n},1]. \\
		\end{cases}
	\end{align*}
	That is, $\mu$ satisfies the strong approximate Shearer inequality with constant $C$.
\end{theorem}
We note that $C \le \min\{\frac{1}{\delta}, en\}$ for any $\delta \in [0,1]$ in Theorem~\ref{mainthm2}.
The request that  $\Gamma$ is positive semidefinite can be met %is no loss of generality since on can always 
by adding a constant diagonal term in the interactions without changing the distribution. On the other hand, the request $\l(\G)\le1-\delta$ is a high temperature or weak dependence condition. The latter is implied for instance by the Dobrushin-type condition
\begin{align}\label{Dobr}
\max_{i\in[n],a\in[q]}\sum_{j\in[n],b\in[q]}|\phi_{i,j}(a,b)|\le1-\delta\,.
\end{align}

Theorem \ref{mainthm2} allows us to derive meaningful estimates in the following mean field models. 
The ferromagnetic mean field $q$-Potts model with parameter $\beta\ge0$ is %given by
 \begin{align}\label{MFqPotts}
\mu(\si) \propto \exp\left(\textstyle{\frac{\beta}{n}\sum_{i,j\in V}}\IND_{\si_i=\si_j}+\textstyle{\sum_{i\in V}}\psi_i(\si_i)\right)\,.
\end{align}

\begin{corollary}\label{cor2}
	Let $\mu$ be the ferromagnetic 
	mean field 
$q$-Potts  model with $\b \in[0, 1)$. Then %, uniformly in the choice of the single site potentials $\psi_i:[q]\mapsto\bbR$,  
	$\mu$ satisfies the strong approximate Shearer inequality with constant 
	%$C:=\min\{ \frac{1}{1-\beta}, c_0 n \}$ where $c_0 > 0$ is an absolute constant.
	$C:=\left(1-\b\right)^{-1}$.
\end{corollary}
\begin{proof}
By adding a harmless diagonal term we can write \[\G(i,a;j,b):= \frac{\beta}n \IND_{i=j}\IND_{a=b}+\IND_{i\neq j}\frac{\beta}n \IND_{a=b}.\] Then $\G$ is positive semidefinite and its  maximum eigenvalue equals  $\b$.   Thus the result follows from Theorem \ref{mainthm2}.
\end{proof}

In the case $q=2$, \eqref{MFqPotts} is referred to as the mean field Ising or Curie-Weiss model, and Corollary \ref{cor2} covers the whole high-temperature region since $\b_c=1$ is known to be the critical value of the parameter $\b$ in this case. Moreover, taking $\delta=0$, Theorem \ref{mainthm2} predicts a constant $C=O(n)$ at criticality. 
In fact, we are able to obtain a better constant $C=O(\min\{(1-\b)^{-1}, \sqrt{n}\})$ than Corollary \ref{cor2},
%Moreover, in this case the constant $C=(1-\b)^{-1}$ 
where the bound $O(\sqrt{n})$ is tight for the critical Curie-Weiss model with $\b=1$, as observed in \cite{LLP10,DLP09,BD2022,CCYZ25} for the Glauber dynamics. On the other hand, for large $q$ the criterium in Corollary \ref{cor2} does not cover the whole high-temperature region since it is known that $\b_c$ grows like $\log q$ \cite{PottsCW}.

\begin{corollary}\label{cor2-Ising}
	Let $\mu$ be the ferromagnetic 
	mean field 
	Ising model with $\b \in[0, 1]$. Then %, uniformly in the choice of the single site potentials $\psi_i:[q]\mapsto\bbR$,  
	$\mu$ satisfies the strong approximate Shearer inequality with constant 
	\begin{align*}
		C:=\min\left\{ (1-\b)^{-1}, c_0 \sqrt{n} \right\},
	\end{align*}
	where $c_0 > 0$ is an absolute constant.
	%$C:=\left(1-\b\right)^{-1}$.
\end{corollary} 

 The $q$-Potts spin glass %$q$-Sherrington--Kirkpatrick model 
 with parameter $\beta\ge 0$ is the random probability measure %given by
 \begin{align}\label{muSK}
\mu(\si) \propto \exp\left(\textstyle{\frac{\beta}{\sqrt n}\sum_{i,j\in V}}J_{i,j}\IND_{\si_i=\si_j}+\textstyle{\sum_{i\in V}}\psi_i(\si_i)\right)\,,
\end{align}
where $J_{i,j}=J_{j,i}$, and $(J_{i,j},\,i\le j)$ are i.i.d.\ standard normal random variables. 
\begin{corollary}\label{cor3}
	Let $\mu$ be the $q$-Potts spin glass %$q$-Sherrington--Kirkpatrick 	model 
	with parameter $\b \in[0,1/4)$. Then, for any $\d>0$, with high probability as $n\to\infty$, %, uniformly in the choice of the external fields $\psi_i:[q]\mapsto\bbR$, 
	$\mu$ satisfies the strong approximate Shearer inequality with constant $C:=(1+\d)\left(1-4\b\right)^{-1}$.
\end{corollary}
\begin{proof}
We recall that for any $\e>0$, the GOE random matrix $n^{-1/2}J_{i,j},\,i,j\in[n]$, with high probability as $n\to\infty$, has eigenvalues in the interval $[-2-\e,2+\e]$. Therefore, adding the diagonal term $\b(2+\e)\IND_{i=j}\IND_{a=b}$ one obtains the matrix 
\[\G(i,a;j,b):= \b(2+\e)\IND_{i=j}\IND_{a=b}+\b n^{-1/2}J_{i,j}\IND_{i\neq j} \IND_{a=b},\] with eigenvalues in $[0,\b(4+2\e)]$.
 Thus the result follows from Theorem \ref{mainthm2} by choosing $\e$ sufficiently small depending on $\d>0$ and $1-4\b>0$.
\end{proof}
When $q=2$ the Potts spin glass is called the Sherrington--Kirkpatrick model. In this case, for the same regime $\b<1/4$, tight bounds for the Glauber dynamics where obtained in \cite{EKZ2022,AJKPV2021,cheneldan}. Corollary \ref{cor3} uses essentially the same technique of \cite{cheneldan} to extend this to arbitrary block dynamics and to an arbitrary number of colors $q$. 

Next, we consider the Potts model on a graph $G$, that is the model defined by \eqref{spinsys} with the choice \eqref{eq:Potts}. Reasoning as in the above corollaries, the criterion in Theorem \ref{mainthm2} is easily seen to imply that if the parameter $\b\in\bbR$ is such that  $|\b|< (2\Delta)^{-1}$ then %, uniformly in the choice of the self potentials $\psi_i:[q]\mapsto\bbR$,  
	$\mu$ satisfies the strong approximate Shearer inequality with constant $C:=\left(1-2|\b|\D\right)^{-1}$, where $\D$ is the maximum degree of the graph $G$.  However, for graphical models much more 
precise results can be obtained by using the notion of spectral independence introduced in \cite{ALO}, together with results from \cite{CLV21,FGYZ2022} establishing the spectral independence property for such systems. In particular, for the Ising model
on a graph $G$, defined by \eqref{spinsys} with the choice \eqref{eq:Potts} and $q=2$, using the stochastic localization technique as in \cite{cheneldan}, we obtain the following result. Say that the system is in the {\em tree uniqueness regime with gap} $\d\in(0,1)$ 
if $\b\in\bbR$ satisfies
\begin{align}\label{eq:tU}
\exp(2|\b|) \le \frac{\D-\d}{\D-2+\d}
\quad \Leftrightarrow \quad
\tanh(|\beta|) \le \frac{1-\delta}{\Delta-1}
\,,
\end{align}
where $\D$ is the maximum degree of the graph $G$. 
\begin{theorem}\label{th:SpInd}
For the Ising model on a graph $G$ of maximum degree $\D\ge 3$, with $\b\in\bbR$ in the tree uniqueness regime with gap $\d\in(0,1)$, 
the strong approximate Shearer inequality holds with constant $C_\d=\mathrm{poly}(1/\d)$. In particular,  for any $\a\in\cW$, the $\a$-Gibbs sampler has mixing time $C_\d\, \g(\a)^{-1}\times O(\log n)$. 
%$C_\d=e^{O(1/\d)}$
\end{theorem}
Theorem \ref{th:SpInd} was established in \cite{BCCPSV22} with $C_\d$ replaced by a constant growing exponentially in $\D$, thus providing tight bounds for graphs with bounded degrees. Here, there is no dependence on $\D$ for the constant $C_\d$ and therefore the result yields the desired bounds in the setting of graphs with unbounded degrees. %provided the spin system is in the uniqueness regime. 
The above result is optimal since it is known that outside of the uniqueness regime one cannot expect such rapid mixing estimates to hold  \cite{SS14}. If one restricts to the Glauber case of one-point factorizations, that is for $\a_A=\frac1n\IND_{|A|=1}$, the above estimate is established in \cite{cheneldan} with a constant $C_\d=\mathrm{exp}(1/\d)$; see also \cite{MosselSly} for an earlier proof of $O(n\log n)$ mixing of Glauber dynamics for Ising models in the tree uniqueness regime. 
Finally, we note that a related result can be obtained for the  Potts model, that is with $q\ge 3$; see Remark \ref{rem:SpIndPotts} below. However, unlike the Ising
model, for the Potts model the bounds do not reach the tree uniqueness threshold, and do not provide an optimal $\D$-dependence. 

When $|\beta| = \beta_c(\Delta) := \frac{1}{2} \ln(\frac{\Delta}{\Delta-2})$, i.e., for the critical (either ferromagnetic or antiferromagnetic) Ising model, using a refined covariance estimate from \cite{CCYZ25}, we also establish the strong approximate Shearer inequality. As in Theorem~\ref{mainthm2}, the constant $C = C_n$ in the inequality is a polynomial in $n$ rather than an absolute constant.
We note that this is necessary as the inverse spectral gap of the Glauber dynamics in this case is super linear, see \cite{CCYZ25}.
\begin{theorem}\label{th:critical-Ising}
	For the critical Ising model on a graph $G$ of maximum degree $\D\ge 3$, with $|\b| = \beta_c(\Delta):= \frac{1}{2} \ln(\frac{\Delta}{\Delta-2})$,
	the strong approximate Shearer inequality holds with constant $C_n = O(n^{1+\frac{2}{\Delta-2}})$. In particular,  for any $\a\in\cW$, the $\a$-Gibbs sampler has mixing time $C_n\, \g(\a)^{-1}\times O(\log n)$. 
\end{theorem}

For the critical Ising model, the constant $C_n$ from the strong approximate Shearer inequality is at best $\Omega(\sqrt{n})$, 
which was shown in \cite{CCYZ25} for the Glauber dynamics and is witnessed by
an infinite sequence of graphs of maximum degree at most $\Delta$, for any $\Delta \ge 3$.
The $\Omega(\sqrt{n})$ bound is tight for the critical mean field Ising model (i.e., $G$ is a complete graph and $\Delta = n-1$), as discussed before Corollary~\ref{cor2-Ising}.
It is thus tempting to ask if one could improve the constant in Theorem~\ref{th:critical-Ising} to $C_n = O(\sqrt{n})$.
%, even just for the Glauber dynamics.

%\begin{corollary}\label{cor1}
%	Let $\mu$ be
%	 the ferromagnetic %mean field 
%$q$-Potts  model on a finite graph $G=(V,E)$. If $\b < (2\Delta)^{-1}$ then %, uniformly in the choice of the self potentials $\psi_i:[q]\mapsto\bbR$,  
%	$\mu$ satisfies the strong approximate Shearer inequality with constant $C:=\left(1-2\b\D\right)^{-1}$.
%\end{corollary}
%
%\begin{proof}
%To ensure that $\G$ is positive semidefinite we add a constant diagonal term  $\phi_{i,i}(a,b):=\b\D\IND_{a=b}$ to the interaction, which does not alter the distribution. Thus the matrix $\G$ in \eqref{def:Gamma} becomes $\G(i,a;j,b):= \beta\Delta \IND_{i=j}\IND_{a=b}+\beta A(G)_{i,j}\IND_{a=b}$. The maximum eigenvalue of this matrix can be easily bounded by $2\b\D$ and thus the result follows from Theorem \ref{mainthm2}.
%\end{proof}

\subsection{Generalizations to a multicomponent system}
Here we consider an extension of Theorem \ref{mainthm2} to a spin system consisting of several weakly interacting components, each of which is a generic spin system. %An application of the result 
%to a model that arises naturally as the linear mean field approximation for a nonlinear dynamics will be discussed.

Let $n\in\bbN$ be given and suppose that for each $i=1,\dots,n$ we have a probability measure $\mu_i\in\cP(\O_i)$, where $\O_i=[q_i]^{L_i}$, where $q_i\geq 2$ and $L_i\ge 1$ are given integers. The measures $\mu_i$ can be interpreted as spin systems, but they do not need to have the form \eqref{spinsys}, and are in fact allowed to be arbitrary elements of $\cP(\O_i)$. For instance, one may have distributions with hard constraints.
We call $\mu_i$ the $i$-th component.  A configuration of the $i$-th component is a spin configuration $\si_i = (\si_i(1),\dots,\si_i(L_i))$. Next we add some interaction between the components and within the components. 
Thus, we consider the spin system $\mu\in\cP(\O)$, where $\O=\times_{i=1}^n\O_i$, defined by
 \begin{align}\label{spinsysgen}
\mu(\si) \,\propto\, \exp\left(\textstyle{\sum_{(i,\ell),(j,k)}}\varphi_{i,\ell;j,k}(\si_i(\ell),\si_j(k))+\textstyle{\sum_{(i,\ell)}}\psi_{i,\ell}(\si_i(\ell)\right)\prod_{i=1}^n\mu_i(\si_i)\,,
\end{align}
for some pair potentials $\varphi_{i,\ell;j,k}:[q_i]\times[q_j]\mapsto \bbR$, satisfying the symmetry $ \phi_{i,\ell;j,k}(a,b)=\phi_{j,k;i,\ell}(b,a)$ and for some self potentials $\psi_{i,\ell}:[q_i]\mapsto\bbR$. Here and below it is understood that the sum over $(i,\ell)$ is such that $i$ runs from $1$ to $n$, and for each fixed $i$, $\ell$ runs from $1$ to $L_i$, and similarly for the sums over $(j,k)$.
Define the interaction matrix
\begin{align}\label{def:Gammagen}
\G((i,\ell),a;(j,k),b) = \phi_{i,\ell;j,k}(a,b)\,.
\end{align}
Here $i,j\in[n],\ell\in[L_i],k\in[L_j]$ and $a\in[q_i],b\in[q_j]$. Thus, $\G$ is a $M\times M$ symmetric matrix, where $M:=q_1L_1 + \cdots q_nL_n$.  We write $\l(\G)$ for the maximum eigenvalue of $\G$. We also need the following notion of covariance for the component measures $\mu_i$. Let $\mu_i^\psi\in\cP(\O_{i})$ denote the probability measure \[
\mu_i^\psi(\si_i)\propto \mu_i(\si_i)\exp\big({\textstyle{\sum_\ell}\psi_{i,\ell}(\si_i(\ell))}\big),\] where $\psi_{i,\ell}:[q_i]\mapsto\bbR$, $\ell\in[L_i]$,  are arbitrary self-potentials. We call $\Sigma^\psi_i(\ell,a;k,b)$, $\ell,k\in[L_i]$, $a,b\in[q_i]$, the $q_iL_i\times q_iL_i$ matrix defined by
\begin{align}\label{def:R_i}
\Sigma_i^\psi(\ell,a;k,b) = \mu_i^\psi(\IND_{\{\si_i(\ell)=a\}}\IND_{\{\si_i(k)=b\}}) - \mu_i^\psi(\IND_{\{\si_i(\ell)=a\}}) \mu_i^\psi(\IND_{\{\si_i(k)=b\}})  \,.
\end{align}
Letting $\l(\Sigma_i^\psi)$ denote the maximum eigenvalue of $\Sigma_i^\psi$ we define 
\begin{align}\label{def:R}
R:=2\max_i\sup_{\psi}\l(\Sigma_i^\psi).
\end{align}  It is not hard to check that when $\mu_i$ is a product measure over $[q_i]^{L_i}$, then $R=1$, see Lemma \ref{lem:Rlemma} below. 
 The first result allows us to decouple the $n$ components provided the interaction is suitably small compared to $R$. We write again $\cW$ for the set of probability vectors $\a=(\a_A,\,A\subset[n])$, and use the notation $\ent_A f$ for the entropy of $f$ with respect to the conditional distribution $\mu_A$ obtained from $\mu$, by freezing all component configurations $\si_j$, $j\notin A$.  
  
\begin{theorem}\label{mainthm3} 
	Suppose that $\Gamma$ is positive semidefinite and that $R\l(\G)<1$.	
	Then for all $\alpha\in\cW$, %for any choice of the single components $\mu_i\in\cP(\O_i)$, 
	the spin system $\mu$ in \eqref{spinsysgen} satisfies, 	for all functions $f\ge 0$, 
	\begin{align}\label{sheaineq2s}
		\gamma(\alpha) \,\ent f \leq \left(1-R\l(\G)\right)^{-1}%\frac1{1-\l(\G)}
		\textstyle{\sum_{A \subset [n]}}\alpha_A\,\mu\left(\ent_Af\right)\,.
	\end{align} 
	%That is, $\mu$ satisfies the approximate Shearer inequality with constant $C:=\left(1-\l(\G)\right)^{-1}$.
\end{theorem}
Note that the above inequality is not an approximate Shearer inequality since only factorizations with blocks of the form
$\{\si_i(\ell), \,i\in A, \ell=1,\dots,L_i\}$ are considered. That is, we are not saying anything about factorizations within a single component, but only between components. This is natural since we are not assuming anything about the single components $\mu_i$. We interpret  
\eqref{sheaineq2s} as a factorization of entropy into blocks of components. 

%One motivation for Theorem \ref{mainthm3}  is that an application to a model where each 
As an example, suppose each $\mu_i$ is a conservative spin system with $q_i\equiv2$ of the form of a product of Bernoulli measures conditioned on their sum. In this case, we show that $R\le 2$; see Lemma \ref{lem:Rlemma2}. Then, an application of Theorem \ref{mainthm3} shows that as soon as $\l(\G)<1/2$ one can factorize the entropy into blocks of components. More generally, for any fixed number of colors $q_i\equiv q$, if the $\mu_i$ are each of the form of a product of arbitrary single spin distributions conditioned on the number of occurrences for each color $a\in[q]$,  using results for partition-constrained strongly Rayleigh distributions recently derived in \cite{alimohammadi2021fractionally}, one can show that $R\le R_0(q)$, where $R_0(q)$ is independent of the systems sizes $L_i$, and therefore one can obtain factorizations for weakly interacting components in such cases, as soon as $\l(\G)<1/R_0(q)$. %, see Corollary \ref{cor:apply} below for more details. 
These facts can be used to obtain tight entropy production estimates for a particle system approximating the nonlinear Kac-Boltzmann type dynamics for the Ising model studied in \cite{CapSin}. The detailed application to such nonlinear dynamics will be discussed elsewhere.  

The next result addresses the question of a full Shearer-type bound under the assumption that the single components satisfy such an estimate. 
%For any $i\in[n]$, we say that the component $\mu_i$ satisfies the strong approximate Shearer inequality with constant $C_i$ if the tilted measure $\mu_i^\psi(\si_i) \propto\mu_i(\si_i)e^{\sum_{\ell}\psi_\ell(\si_i(\ell))}$
%satisfies the approximate Shearer inequality with constant $C_i$ uniformly in the choice of single site potentials $\psi=(\psi_\ell)$, $\psi_\ell:[q_i]\mapsto\bbR$. Similarly, we say that the full system $\mu$ in \eqref{spinsysgen}
%satisfies the strong approximate Shearer inequality with constant $C$ if the tilted measure $\mu^\psi$ obtained by adding a self-potential term $\sum_{(i,\ell)}\psi_{i,\ell}(\si(\ell))$ to the exponential in   \eqref{spinsysgen}
%satisfies the approximate Shearer inequality with constant $C$ uniformly in the choice of $\psi=(\psi_{i,\ell})$, $\psi_{i,\ell}:[q_i]\mapsto\bbR$.
\begin{theorem}\label{mainthm4} 
	Suppose that $\Gamma$ is positive semidefinite and that $R\l(\G)<1$. Suppose further that each component $\mu_i$ satisfies the strong approximate Shearer inequality with constant $C_i$. 	
	Then the system $\mu$ in \eqref{spinsysgen} satisfies the strong approximate Shearer inequality with constant $C:=\left(1-R\l(\G)\right)^{-1}\max_iC_i$. Moreover, the above statement continues to hold if we replace $R$ with $\max_iC_i$. 
%	\begin{align}\label{sheaineq2}
%		\gamma(\alpha) \,\ent f \leq \left(1-\l(\G)\right)^{-1}%\frac1{1-\l(\G)}
%		\sum_{A \subset V}\alpha_A\,\mu\left(\ent_Af\right)\,,
%	\end{align}
%	holds for all functions $f\ge 0$. 
%	%That is, $\mu$ satisfies the approximate Shearer inequality with constant $C:=\left(1-\l(\G)\right)^{-1}$.
\end{theorem}
Theorem \ref{mainthm4} is a generalization of Theorem \ref{mainthm2}, since if $\mu_i$ is a product measure over $[q_i]^{L_i}$ for each $i$, then we know that %$R=1$ (see Lemma \ref{lem:Rlemma}) and that 
each component satisfies the Shearer inequality with $C_i\equiv 1$ (see the comment after Definition \ref{def:ent_fact}), and thus, in the special case where $q_i\equiv q$,  we recover the statement of  Theorem \ref{mainthm2}. The fact that the statement of Theorem \ref{mainthm4} holds also with $R$ replaced by  $\max_iC_i$ is a consequence of the estimate $R\le\max_iC_i$, which we establish in Lemma \ref{lem:Rlemma} below. 

We remark that Theorem \ref{mainthm4} is nontrivial even in the case of non-interacting components, that is when $\G\equiv 0$. We refer to \cite[Lemma 3.2]{CP2021} for a closely related tensorization lemma in the non-interacting case. By allowing some weak interaction between %and within 
the components, here we establish a considerable generalization of that result.
%We note that bounds on $R$ can be obtained if some correlation inequalities are known for the components $\mu_i$. For instance, in  Lemma \ref{lem:Rlemma2} we discuss the case of product measure conditioned on their sum. 
%Moreover, it is not difficult to check that when each component $\mu_i$ satisfies the strong approximate Shearer inequality with constant $C_i$, one has $R\le\max_iC_i$, see Lemma \ref{lem:Rlemma} below. 
As an example of application, suppose that each component $\mu_i$ is a Curie-Weiss model as in \eqref{MFqPotts}, with $\beta<1$. Then, from Corollary \ref{cor2} we know that each $\mu_i$ satisfies the strong approximate Shearer inequality with constant $C_1=(1-\b)^{-1}$. %Moreover, it is not difficult to check that $R=O((1-\b)^{-1})$ in this case, see Lemma \ref{lem:Rlemma} below. 
Thus, from Theorem \ref{mainthm4} we deduce that a collection of weakly interacting Curie-Weiss models also satisfies the strong approximate Shearer inequality with a constant $C=C_1(1-C_1\l(\G))^{-1}$ provided $C_1\l(\G)<1$.

\section{Proof of Theorem \ref{mainthm2}, Theorem \ref{th:SpInd}, and Theorem \ref{th:critical-Ising}}

\subsection{Reduction to binary spins}\label{sec:reduction}
The first step is a change of variables which maps the spin system on $\O=[q]^n$ to a spin system on $\{0,1\}^{qn}$.
We define a new single spin space $\hat{\Om}_0 = \left\{x\in\{0,1\}^{q}:\, \sum_{a=1}^qx_{a} =1\right\}$. Clearly, $[q]$ is in one to one correspondence with $\hat{\Om}_0$
and $\O$ is in one to one correspondence with $\Om_0:=(\hat{\Om}_0)^n$. To formalize this, we use the map $\O\ni\si=(\si_1,\dots,\si_n) \mapsto \eta=(\eta^1,\dots,\eta^n)$, where each $\eta^i\in\hat{\Om}_0$ is such that $\eta^i_a = \IND_{\si_i=a}$ for all $a\in[q]$. 
The uniform distribution over $\O$ is equivalent to the uniform distribution over $\O_0$. 
For $x,y\in\bbR^{qn}$ we use the notation $x=(x^i_a, i\in[n], a\in[q]), y=(y^i_a, i\in[n], a\in[q])$,  and 
\[
\scalar{x}{y} = \sum_{i=1}^n\sum_{a=1}^qx^i_ay^i_a.
\]
The probability measure $\mu$ from \eqref{spinsys}, in the new variables $\eta\in\O_0$ becomes the probability $\r\in \cP(\O_0)$ given by
\begin{align}\label{isingmatr}
	\rho(\eta) \propto \exp\left(\scalar{\eta}{\G\eta}+\scalar{h}{\eta}\right)\,\ind_{\{\eta \in \Om_0\}}\,,
\end{align}
where $\G(i,a;j,b) = \phi_{i,j} (a,b)$ is the matrix defined in \eqref{def:Gamma},  and $h\in\bbR^{qn}$, defined by $h^i_a := \psi_i(a)$, are the new external fields. 
The advantage of this expression over \eqref{spinsys} is that the interaction is given by a quadratic form and the self potentials now act linearly in the spin variables $\eta^i_a$. The price to pay is that the uniform distribution over $\O_0$ is no longer a product measure in the new spin variables $\eta^i_a$. However, it is still a product over the ``block'' variables $\eta^i=(\eta^i_a,\,a\in[q])$ and this will suffice for our analysis to go through. 

\subsection{Stochastic localization process}\label{sec:stloc}
In the new variables $\eta\in\O_0$, we apply the stochastic localization process introduced in \cite{eld13}, see also \cite{cheneldan}. The crucial properties of this construction are discussed in detail in these references. However, for the sake of completeness we include a short proof of some of the main statements adapted to our setting. 
Let $(B_t, \,t\ge0)$ be a standard Brownian motion in $\bbR^{qn}$ adapted to a filtration $(\mathcal{F}_t, \,t\ge0)$.  
 Consider the system of differential equations
 \begin{align}\label{locprocmarta}
		dF_t(\eta) = F_t(\eta)\scalar{\eta-m(F_t )}{(2\G)^{\frac12}dB_t}\,, \qquad \eta\in\O_0\,,
	\end{align}
	with initial condition $F_0(\eta)=1$ for all $\eta\in\O_0$. 	Here $m(F_t)\in\bbR^{qn}$ denotes the linear function
	\begin{align}\label{eq:mt}
	m(F_t):=\sum_{\eta \in \O_0}\eta\, F_t(\eta)\r(\eta)\,. 
	\end{align}
Existence and uniqueness of the solution for the system \eqref{locprocmarta} can be derived by standard methods, see \cite{eld13}. The next lemma summarizes some key features of the process.
	
	\begin{lemma}\label{lem:ito}
	The solution $F_t$ to \eqref{locprocmarta} is an $(\cF_t)$-martingale satisfying, for all $\eta\in\O_0$,  
	\begin{align}\label{locprocmart}
	F_t(\eta) = Z^{-1}_t\exp\left( - \tfrac12\scalar{\Sigma_t\eta}{\eta} + \scalar{y_t}{\eta}\right)\,, \ \ \ \ \  t \geq 0,
\end{align}
where $\Sigma_t := 2t\G ,\; y_t := (2\G)^{\frac12}B_t + 2\Gamma\int_0^tm(F_s)ds$, and $Z_t$ is the normalizing constant
\begin{align}\label{eq:zt}
	Z_t = \sum_{\eta\in \Om_0}\rho(\eta)\exp\left( - \tfrac12\scalar{\Sigma_t\eta}{\eta} + \scalar{y_t}{\eta}\right)\,.
\end{align}
In particular, $\r_t:=F_t \r$ is a probability measure on $\O_0$ for all $t\ge 0$, and $\r_t(\eta)$ is a martingale for each fixed $\eta$. 
	\end{lemma}
\begin{proof}	
The martingale property is immediate from the definition \eqref{locprocmarta}. Define $G_t(\eta): = \exp\left(-\frac12 \scalar{\Sigma_t \eta}{\eta} + \scalar{y_t}{\eta}\right)$ and $H_t:=Z_t^{-1}G_t$, where $Z_t$ is given in \eqref{eq:zt}. Using It\^o formula, one checks that
	\begin{align*} %\label{diff1}
		&dG_t(\eta) = G_t(\eta)\scalar{\eta}{2\G m(F_t)dt+(2\G)^{\frac12}dB_t},\\ %\label{diff2}
		&dZ_t^{-1} =-Z_t^{-1}\scalar{m(F_t)}{(2\G)^{\frac12}dB_t}. %+ Z_t^{-1}\scalar{2\G m(\rho_t)}{m(\rho_t)}dt,
	\end{align*}
	It follows that
	\begin{align*}
dH_t  
		&=dG_t(\eta)Z_t^{-1} + G_t(\eta)dZ_t^{-1} + dG_t(\eta)dZ_t^{-1}\\
				&= F_t(\eta)\scalar{\eta-m(F_t)}{(2\G)^{\frac12}dB_t},
	\end{align*}
	and therefore $H_t$ and $F_t$ satisfy the same equation \eqref{locprocmarta}. This proves \eqref{locprocmart}. In particular, $\r_t:=F_t \r\in\cP(\O_0)$ for all $t\ge 0$, and $\r_t(\eta)$ is a martingale for each fixed $\eta$. 
	\end{proof}	
	The above construction yields a process $(F_t, t\ge 0)$ which depends on the choice %can be applied to any 
	of the distribution $\r\in \cP(\O_0)$ through the definitions \eqref{locprocmarta}-\eqref{eq:mt}. The process $(\r_t, t\ge 0)$ is called a stochastic localization process for $\r$. 
When %apply it with 
$\r$ is given by the expression \eqref{isingmatr}, the identity \eqref{locprocmart} shows that  $\r_t\in\cP(\O_0)$ satisfies
\begin{align}\label{stoclocmeasure}
	\rho_t(\eta) \propto \exp\left((1-t)\scalar{\eta}{\Gamma\eta} + \scalar{y_t+h}{\eta}\right)\IND_{\{\eta \in \Om_0\}}\,, \ \ \ \ \ \ t \geq 0.
\end{align}
This defines a {\em random} probability measure on $\O_0$ for each $t\ge 0$. When $t=1$, we have a random product measure over the block variables $\eta^i=(\eta^i_a,\,a\in[q])$, $i\in[n]$:
 \begin{align}\label{prodmeasure}
	\rho_1(\eta)=\prod_{i=1}^n\rho_{1,i}(\eta^i)\,,\qquad  \rho_{1,i}(\eta^i)\propto \exp\left(\scalar{y_1^i+h^i}{\eta^i}\right)\IND_{\{\eta^i \in \hat\Om_0\}}\,,
\end{align}
where $y_1^i$ and $h^i$ represent the $i$-th block of the variables $y_1,h\in\bbR^{qn}$, respectively. 

\subsection{Entropic stability}
Next, we turn to a key estimate established in \cite{cheneldan}, which allows us to compare relative entropies with respect to $\r_t$ with relative entropies with respect to $\r$ by using a covariance estimate. 
We use the notation \begin{align}\label{eq:covtv}
\cov[\nu](i,a;j,b) = \nu(\eta^i_a\eta^j_b)- \nu(\eta^i_a)\nu(\eta^j_b),
\end{align}
for the $qn\times qn$ covariance matrix of any  $\nu\in\cP(\O_0)$, and  write $\l(\cov[\nu])$ for its largest eigenvalue. If $v \in \bbR^{qn}$, the tilted measure $\mathcal{T}_v\nu$ is defined as 
\begin{align}\label{eq:tvnu}
\mathcal{T}_v\nu(\eta) \propto\nu(\eta)e^{\scalar{v}{\eta}}\,.
	\end{align}
The next lemma is a version of \cite[Proposition 39]{cheneldan}.
\begin{lemma}\label{prop1}
	Suppose there exist $T > 0$ and a deterministic function $\kappa$ such that a.s.\ $\l(\cov[\cT_v\rho_t]) \leq \kappa(t)$ for all $t \in [0,T],$ for all $v\in\bbR^{qn}$.
	Then,  for all functions $f\ge 0$,
	\begin{align}
		\bbE\left[\ent_{\rho_T}f\right] \geq %e^{-2\l(\G)\int_0^T\kappa(s)ds}
		e^{-\int_0^Ta(s)ds}\,\ent_\rho f,\qquad a(s):=2\l(\G)\k(s)\,,
	\end{align}
	where $\bbE$ denotes the expectation over %with respect to the Brownian motion underlying 
	the stochastic localization process $\r_t$. 
	
\end{lemma}

\begin{proof}%[Proof of Lemma \ref{prop1}]
It suffices to show that $e^{\int_0^ta(s)ds}\ent_{\rho_t} f$, $t\in[0,T]$ is a submartingale. In turn, the It\^o formula shows that this is implied by the differential inequality
\begin{align}\label{eq:topro}
		d\,\ent_{\rho_t} f \ge- a(t) \ent_{\rho_t}f + \text{martingale} \ \ \ \ \ \forall t \in [0,T]\,.
	\end{align}
To prove \eqref{eq:topro}, following \cite{cheneldan}, we consider the martingale $M_t := \rho_t(f)$, and observe that 
since $F_t$ solves \eqref{locprocmarta}, one has
	\begin{align*}
		dM_t = M_t\scalar{\left(m(U_t)-m(F_t)\right)}{(2\G)^{\frac12}dB_t},
	\end{align*}
	where 
$U_t(\eta):=F_t(\eta)f(\eta)/M_t$. Note that $U_t\r\in\cP(\O_0)$. 
	Next, It\^o formula shows that %and the previous equation one has
	\begin{align}
		d\left(M_t \log M_t\right) = \frac{1}{2}M_t\|(2\G)^{\frac12}\left(m(U_t)-m(F_t)\right)\|_2^2\,dt + \text{martingale}.
	\end{align}
	Therefore,
	\begin{align} \nonumber
		d\,\ent_{\rho_t}(f) &= \sum_{\eta \in \O_0} d\rho_t(\eta)f(\eta)\log f(\eta) - d\left(M_t \log M_t\right)  \\ \label{diffent}
		& = -\frac{1}{2}M_t\|(2\G)^{\frac12}\left(m(U_t)-m(F_t)\right)\|_2^2\,dt + \text{martingale}.
	\end{align}
	Since $\l(\cov[\cT_v\rho_t])\leq \kappa(t)$,  for all $t \in [0,T]$ and for all $v\in\bbR^{qn}$, we can apply the general entropic stability estimates established %in Lemma 40 %and Lemma 31 
	in \cite[Lemma 31]{cheneldan} and \cite[Lemma 40]{cheneldan} to infer
	\begin{align*}
		\frac{1}{2}\|(2\G)^{\frac12}\left(m(U_t)-m(F_t)\right)\|_2^2\le 2\l(\G)\k(t)\,H(U_t\r\,|\,F_t\r)\,.
	\end{align*}
	Using the previous inequality in \eqref{diffent} and noting that $M_tH(U_t\r\,|\,F_t\r)=\ent_{\rho_t}(f)$ yields \eqref{eq:topro} with the function $ a(t)=2\l(\G)\k(t)$.
\end{proof}

\subsection{Covariance estimate}
To make an effective use of the entropic stability in Lemma \ref{prop1}, we need to control $\l(\cov[\cT_v\r_t])$
uniformly in $v\in\bbR^{qn}.$ We first need the following simple lemma, which covers the case $n=1$.

	\begin{lemma}\label{lem:cov}
		For all $q\in\bbN$, $q\ge 2$, for any probability vector $(m_1,\dots,m_q)$, the  $q\times q$ covariance matrix $C_{a,b} = m_a(1-m_b)\IND_{a=b} - m_am_b \IND_{a\neq b}$ has maximum eigenvalue
		\[ \l(C)\le 1/2.
		\] 
	\end{lemma}
	\begin{proof}
		For any $u\in\bbR^q$ one has
		\[
		\textstyle{\sum_{a,b}}u_aC_{a,b}u_b = \var(\scalar{u}{\o}) = \frac12\textstyle{\sum_{a,b}}m_am_b(u_a-u_b)^2\,,
		\]
		where $\o$ is the random variable that equals ${\rm e}_a$ %, the $i$-th canonical basis vector in $\bbR^q$, 
		with probability $m_a$, $({\rm e}_1,\dots,{\rm e}_q)$ is the canonical basis in $\bbR^q$, and $\scalar{u}{\o}$ denotes the canonical scalar product in $\bbR^q$. 
		Using $(u_a-u_b)^2\le 2(u_a^2+u_b^2)\ind_{a\neq b}$, and $m_a(1-m_a)\leq 1/4$, we obtain
		\[
		\textstyle{\sum_{a,b}}u_aC_{a,b}u_b \leq\textstyle{\sum_{a,b}}m_am_b(u_a^2+u_b^2)\ind_{a\neq b}=2\textstyle{\sum_{a}}m_a(1-m_a)u_a^2\leq \frac12\textstyle{\sum_{a}}u_a^2\,. \qedhere
		\]
	\end{proof}
	We note that the above estimate is tight since if $m_1=m_2=1/2$ one has $\l(C)= 1/2$.
Note also that Lemma \ref{lem:cov} says that, when $n=1$, any covariance matrix of the form \eqref{eq:covtv} must have maximal eigenvalue bounded by $1/2$, since when $n=1$ the distribution $\cT_v\nu$ has covariance $C$ as in the statement of the lemma. The following is a version % generalization 
of  \cite[Lemma 50]{cheneldan}, which in turn was based on the decomposition introduced by Bauerschmidt and Bodineau \cite{BB19}. We adapted the proof to our setting.
\begin{lemma}\label{pottscovlemma}
	Consider the spin system $\rho\in\cP(\O_0)$ given by \eqref{isingmatr}, and assume that $\G$ is positive definite with $\l(\G) \le 1-\delta$ where $\delta \in [0,1)$. Then for any $v \in \bbR^{qn}$, $t\in[0,1]$, one has
	\begin{align}\label{eq:covrov}
\l\(\cov[\mathcal{T}_v\rho_t]\) \leq \min \left\{ \frac1{2(1-(1-\delta)(1-t))}, \frac{n}{2} \right\}.
	\end{align}
\end{lemma}
\begin{proof}
	We first show the trivial upper bound of $n/2$ in \eqref{eq:covrov}, which in fact holds for any measure $\nu\in\cP(\O_0)$.
	Let $\eta$ be generated from a measure $\nu\in\cP(\O_0)$. Then it holds
	\begin{align*}
		\l\(\cov[\nu]\)
		\le 
        %\left\| \cov[\nu] \right\|_\infty	= 
        \max_{i\in[n], a\in [q]} \sum_{j\in [n]} \sum_{b\in [q]} \left| \cov[\nu](i,a;j,b) \right|.
	\end{align*}
	For any $i,j \in [n]$ and $a \in [q]$, we observe that
	\begin{align*}
		\sum_{b\in [q]} \left| \cov[\nu](i,a; j,b) \right|
		&= \sum_{b\in [q]} \left| \nu(\eta^i_a \eta^j_b) - \nu(\eta^i_a) \nu(\eta^j_b)\right|\\
		&= \nu(\eta^i_a) (1 - \nu(\eta^i_a)) \sum_{b\in [q]} \left| \nu(\eta^j_b \mid \eta^i_a = 1) - \nu(\eta^j_b \mid \eta^i_a = 0) \right|\\
		&\le 2\nu(\eta^i_a) (1 - \nu(\eta^i_a)) \le \frac{1}{2}.
	\end{align*}
	It then follows that $\l\(\cov[\nu]\) \le n/2$, as claimed.
	
	We now proceed to prove the first upper bound in \eqref{eq:covrov}.
	From \eqref{stoclocmeasure} we see that it suffices to show the inequality \eqref{eq:covrov} with $t=0$ and with $\mathcal{T}_v\r_t$ replaced by %and with $\mathcal{T}_v\rho$ replaced by 
	the measure 
\begin{align}\label{eq:covrov11}	\rho_v(\eta) \propto\exp\left(\scalar{\eta}{\G\eta})+\scalar{v}{\eta}\right)\IND_{\{\eta \in \Om_0\}}\,,
\end{align}
with arbitrary $v \in \bbR^{qn}$.
%Fix a constant $c>0$. 
A gaussian integration shows that for any $\eta\in\O_0$, 
	\begin{align}\label{eq:covrov12}
		\exp(\scalar{\eta}{\G\eta})  \propto %{\rm const.}\times
		 \int_{\bbR^{qn}}dy\exp\left(-\frac12\scalar{y}{(2\G)^{-1}y} +\scalar{\eta}{y} \right),
	\end{align}
where the implicit positive constant is independent of $\eta$. Define the function
\begin{align*}
		f(y) := \exp\left(-\frac12\scalar{y}{(2\G)^{-1}y} + \log\left(Z\left(y + v\right)\right)\right)\,, %\cdot Z_f^{-1}
	\end{align*}
where $Z\left(w\right):=\sum_{\eta \in \Om_0}\exp\left(\scalar{\eta}{w}\right)$, for any $w\in\bbR^{qn}$.
From \eqref{eq:covrov11} and \eqref{eq:covrov12} we obtain 
\begin{align*}
		\rho_v(\eta) &\propto \int_{\bbR^{qn}}dyf(y)Z\left(y + v\right)^{-1}
\exp\left( \scalar{\eta}{y + v} \right)\ind_{\{\eta \in \Om_0\}}\,.
	\end{align*}
		If $\nu$ denotes the uniform probability measure over $\Om_0$ then \[
		Z\left(y + v\right)^{-1}\exp\left( \scalar{\eta}{y + v} \right)\ind_{\{\eta \in \Om_0\}}=\mathcal{T}_{y + v}\nu(\eta)\,.\] 
		Therefore, 
\begin{align}\label{eq:decrho}
\rho_v(\eta) = \int \mathcal{T}_{y + v}\nu(\eta)\hat f(y)dy\,,\end{align}
where $\hat f:=f/\int dyf(y)$ is a probability density on $\bbR^{qn}$.
This expresses $\r_v$ as a mixture of the tilted measures $\mathcal{T}_{y + v}\nu$. %Letting $X$ be a $\bbR^{qn}$-valued random variable with density $\hat f$, t
By the law of total covariance,
	\begin{align}\label{eq:covtot}
		\cov[\rho_v] &= \cC\,+\,\int   dy\hat f(y) \cov [\mathcal{T}_{y + v}\nu]
		 \,,%\cov\left({\bf E}\left[\mathcal{T}_{c X + v}\nu\right]\right),
	\end{align}
	where %${\bf E}$ stands for the expectation over $X$ and 
	$\cC$ stands for the $qn\times qn$ matrix
	\[
	\cC(i,a;j,b):= \int  w^i_a(y)w^j_b(y) \hat f(y) dy - \(\int  w^i_a(y) \hat f(y) dy\)\(\int  w^j_b(y) \hat f(y) dy\),
	\]
	and we use the notation $w^i_a(y):=\sum_{\eta\in\O_0}\mathcal{T}_{y + v}\nu(\eta)\eta^i_a$, $i\in[n]$, $a\in[q]$. 
	Recalling that $\nu$ is a product over the block variables $\eta^i$, we note that for any $x\in\bbR^{qn}$, $\cT_x\nu(\eta)$ has the form $\prod_{i=1}^n \nu_x^i(\eta^i)$ for some probability measures $\nu_x^i$ on $ \hat{\Om}_0$. Thus $\cov [\mathcal{T}_{x}\nu]$ is a block diagonal matrix with blocks given by the $q\times q$ matrices $\cov [\nu_x^i]$. Any probability measure on $\hat\O_0$ has a covariance matrix of the form $C$ given in Lemma \ref{lem:cov} for some probability vector $(m_1,\dots,m_q)$. Therefore, 
	\begin{align}\label{eq:covbo1}
\l(\cov [\mathcal{T}_{x}\nu])\le 1/2\,,\qquad x\in\bbR^{qn}\,.
\end{align} 
%for any $x\in\bbR^{qn}$. 
It follows that the last term in \eqref{eq:covtot} has maximum eigenvalue bounded by $\frac12$. 
	 
We turn to a bound on the matrix $\cC$. With the notation introduced above, for any $u \in\bbR^{qn}$ we write
	\begin{align}\label{eq:covy}
		\scalar{u}{\cC u} &= \textstyle{\sum_{i,j}\sum_{a,b}}u^i_au^j_b\cC(i,a;j,b)=\var\left(g_u(X)\right)\,, 
			\end{align}
where $X$ is the $\bbR^{qn}$-valued random variable with density $\hat f$, %$\var$ is the variance functional, 
and %we write 
$g_u(y):=\sum_{i,a} u^i_{a}  w^i_a(y)$. 
We are going to use the Brascamp-Lieb inequality (Theorem 4.1 in \cite{brasclieb}). %show that $\hat f$ is log-concave. 
Namely, suppose that %we claim that 
$-{\rm Hess}\log \hat f (y)$ is lower bounded, as quadratic forms,  by $h{\rm Id}$ where $h>0$ and ${\rm Id}$ is the $qn\times qn$ identity matrix. Then the Brascamp-Lieb inequality %says that  
	\begin{align}\label{eq:BLineq}
	\var\left(g_u(X)\right)\leq h^{-1}\int dy \hat f(y)\|\grad g_u(y)\|_2^2 \,,
	\end{align}
	where $\grad g_u(y)\in\bbR^{qn}$ denotes the vector $[\grad g_u(y)]^i_a = \frac{\partial}{\partial y^i_a}g_u(y)$. By construction,
	\[
	[\grad g_u(y)]^i_a = \sum_{b=1}^q u^i_{b}  \frac{\partial w^i_b(y)}{\partial y^i_a}= \sum_{b=1}^qu^i_{b} \cC^i_{a,b}\,, %=c^2 \,(C^iu^i)_a\,.\
	\] 
where, for a fixed $i\in[n]$, $\cC^i = (\cC^i_{a,b})$ is a matrix as in Lemma \ref{lem:cov}, with $m_a := w^i_a$. Therefore, the maximum eigenvalue of $(\cC^i)^2$ is bounded by $1/4$ and we obtain, for all $y\in\bbR^{qn}$,  
\[
	\|\grad g_u(y)\|_2^2=\sum_{i,a} ([\grad g_u(y)]^i_a)^2 = \sum_{i,a} (\cC^iu^i)^2_a \le \frac{1}{4}\scalar{u}{u}\,.
	\] 
From \eqref{eq:BLineq} we conclude 	
\begin{align}\label{eq:BLineq2}
	\scalar{u}{\cC u} \leq \frac{1}{4h} \scalar{u}{u}\,\,.
	\end{align}
It remains to establish the bound on the Hessian of $\log \hat f$. We note that 
	\begin{align*}
-\log \hat f(y)= \frac{1}2\scalar{y}{(2\G)^{-1}y} - \log\left(Z\left(y + v\right)\right) + {\rm const.}
	\end{align*}
	Since ${\rm Hess}\log\left(Z\left(y + v\right)\right) =  \cov[\mathcal{T}_{y + v}\nu]$, using Lemma \ref{lem:cov}, as quadratic forms one has
	 ${\rm Hess}\log\left(Z\left(y + v\right)\right)\le  \frac12\,{\rm Id}$. Moreover, ${\rm Hess}\scalar{y}{(2\G)^{-1}y}\ge \l(\G)^{-1}{\rm Id}$. Therefore we obtain the quadratic form bound
		\begin{align}\label{eq:BLineq3}
-{\rm Hess}\log \hat f(y) \ge \frac{1}2
\left(\frac{1}{\l(\G)} - 1\right){\rm Id}.
	\end{align}
	Summarizing, we may take $h =  \frac{1}2
\left(\frac{1}{\l(\G)} - 1\right)$ and  \eqref{eq:BLineq2} becomes $\scalar{u}{\cC u} \leq \frac{\l(\G)}{2(1-\l(\G))}\scalar{u}{u}$. In conclusion, \eqref{eq:covtot} shows that 
$\l\(\cov[\rho_v]\)\le\frac1{2(1-\l(\G))}$. 
\end{proof}

\subsection{Stochastic localization and block dynamics}
The stochastic localization process has been successfully employed in the analysis of Glauber dynamics in \cite{cheneldan}. Here we adapt the arguments from \cite{cheneldan} to handle arbitrary block dynamics. 

Given $\eta\in\O_0$, $t\ge 0$, ad $A\subset [n]$, we write $\r^\eta_{t,A}(\xi) \propto \r_{t}(\xi)\IND_{\xi^{A^c}=\eta^{A^c}}$  for the conditional probability on $\xi^A:=\{\xi^i,\; i\in A\}$
 given the configuration $\eta^{A^c}:=\{\eta^j,\; j\in A^c\}$ outside $A$. Equivalently, we write
 \begin{align}\label{eq:equico}
	\r^\eta_{t,A}(\xi): = \frac{\r_t(\xi)\IND_{\xi^{A^c}=\eta^{A^c}}}{\r_t(\eta^{A^c})},
	\end{align}
	 where $\r_t(\eta^{A^c})=\sum_{\eta^A} \r_t(\eta)$ denotes the marginal of $\r_t$ on the blocks $A^c$.
For each $t\ge0$, $\eta\in\O_0$, $A\subset [n]$, this defines a probability $\r^\eta_{t,A}\in\cP(\O_0)$, and we define the probability kernel
\begin{align}\label{eq:parandom}
Q_{t,A}(\eta,\xi)=%\sum_{A \subset V}\alpha_A
\r_{t,A}^\eta(\xi)\,,\qquad \eta,\xi\in\O_0\,.
\end{align}
 Since $\r=\r_0$, we note that the kernel $P_\a$ defined in \eqref{eq:pa} can be written in the new variables $\eta\in\O_0$ as 
$\sum_{A \subset V}\alpha_AQ_{0,A}.$
%Next, let $\cA$ be a random subset of $[n]$ with distribution given by $\a\in \cW$. For all $t\ge 0$ we consider the random distribution $\r^\eta_{t,\cA}\in\cP(\O_0)$, which equals $\r^\eta_{t,A}$ with probability $\a_A$, $A\subset [n]$.   Note that there are two independent sources of randomness here, namely the one coming from the Brownian motion in the stochastic localization process $\r_t=F_t\r$ and the one inherited from the choice of the random set $\cA$ with law $\a$.  To distinguish them we shall write $\bbE_B$ for the expectation over the first, $\bbE_\a$ 
%  for the expectation over the second, and $\bbE_B\times \bbE_\a$ for the joint expectation.
%  
%For each $t\ge 0$, we define the random probability kernel 
%\begin{align}\label{eq:parandom}
%Q_{t,\cA}(\eta,\xi)=%\sum_{A \subset V}\alpha_A
%\r_{t,\cA}^\eta(\xi)\,,\qquad \eta,\xi\in\O_0\,.
%\end{align}
%From \eqref{eq:pa1} we see that $\bbE_\a Q_{0,\cA}(\eta,\xi)=Q_{\alpha}(\eta,\xi)$. We also write $Q_{t,\cA}f (\eta)= \sum_{\xi}\r_{t,\cA}^\eta(\xi)f(\xi)$, for any function $f$ on $\O_0$. 
The next Lemma is the analogue of  \cite[Proposition 48]{cheneldan} in %will play an important role in 
our setting.

\begin{lemma}%[Proposition 48 in \cite{cheneldan}]
\label{submarlemma}
	For every %$\phi:\Om_0 \longrightarrow \bbR$ and 
	$f:\Om_0 \mapsto \bbR_+$, for any $A\subset V$, % of the random subset of $[n]$,  
	\begin{align}\label{submarent}
		\bbE\left[\textstyle{\sum_{\eta \in \Om_0}}\rho_t(\eta)Q_{t,A}f (\eta) \log Q_{t,A}f (\eta)\right] \geq \textstyle{\sum_{\eta \in \Om_0}}\rho(\eta)Q_{0,A}f (\eta)\log Q_{0,A}f (\eta)
	\end{align}
	where $\bbE$ is the expectation over the random process $\rho_t$.
\end{lemma}

\begin{proof}
%	Suppose that $\cA=A$. Then
%	\[
%	\r^\eta_{t,\cA}(\xi) = \frac{\r_t(\xi)\IND_{\xi^{A^c}=\eta^{A^c}}}{\r_t(\eta^{A^c})},
%	\] where $\r_t(\eta^{A^c})=\sum_{\eta^A} \r_t(\eta)$ denotes the marginal of $\r_t$ on the blocks $A^c$.
 From \eqref{eq:equico} we have
  	\begin{align}\label{eq:submar1}
		\textstyle{\sum_{\eta \in \Om_0}}\rho_t(\eta)Q_{t,A}f (\eta) \log Q_{t,A}f (\eta) =
		\sum_{\eta^{A^c}}\rho_t(f,\eta^{A^c})\log \frac{\rho_t(f,\eta^{A^c})}{\rho_t(\eta^{A^c})},
	\end{align}
	where $\rho_t(f,\eta^{A^c}) := \sum_{\xi\in\O_0}\r_t(\xi)\IND_{\xi^{A^c}=\eta^{A^c}}f(\xi)$. The martingale property of $\r_t$ shows that $\bbE \rho_t(f,\eta^{A^c}) = \rho_0(f,\eta^{A^c})$ and $\bbE \rho_t(\eta^{A^c}) = \rho_0(\eta^{A^c})$. Thus, using
the joint convexity of $(x,y)\rightarrow x\log \frac xy$ and Jensen inequality, we have
	\begin{align}\label{eq:submar2}
		\bbE\left[\rho_t(f,\eta^{A^c})\log \frac{\rho_t(f,\eta^{A^c})}{\rho_t(\eta^{A^c})}\right] \ge 
\rho_0(f,\eta^{A^c})\log \frac{\rho_0(f,\eta^{A^c})}{\rho_0(\eta^{A^c})}\,.
	\end{align}
	The conclusion \eqref{submarent} follows from \eqref{eq:submar1} and \eqref{eq:submar2}.
	\end{proof}

\subsection{Proof of Theorem \ref{mainthm2} and Corollary \ref{cor2-Ising}}\label{sec:proofsec}
\begin{proof}[Proof of Theorem \ref{mainthm2}]
We change variables from $\si\in\O$ to $\eta\in\O_0$ and write $\r$ instead of $\mu$ for our reference Gibbs measure, see \eqref{isingmatr}. We keep the notation $\ent f$ for the entropy of a function $f:\O_0\mapsto \bbR_+$ with respect to $\r$ and $\ent_A f$ for the entropy of $f$ with respect to the conditional distribution $\r_{A}$. 
The proof is complete if we show that 
\begin{align}\label{sheaineq20}
		\sum_{A \subset V}\alpha_A\,\r\left(\ent_Af\right)\ge c_\delta\,\gamma(\alpha)  \,\ent f ,
	\end{align}
holds for all functions $f:\O_0\mapsto \bbR_+$, with the constant $c_\delta=\delta$ for $\delta\in[\frac1n,1]$, and $c_\delta = \frac{1}{ne^{1-\delta n}}$ for $\delta\in [0,\frac{1}{n})$.   

Fix $f\ge 0$ and $\a\in\cW$. 
Using \eqref{eq:deco} and the notation \eqref{eq:parandom} one has
	\begin{align*}
		\textstyle{\sum_{A\subset V}}&\alpha_A\,\rho\left(\ent_A f\right) =\ent f - \textstyle{\sum_{A\subset V}}\alpha_A\ent  \,Q_{0,A}f \\
		&=\textstyle{\sum_{\eta\in\O_0}}\rho(\eta)f(\eta)\log f(\eta) -\textstyle{\sum_{A\subset V}}\alpha_A\,\textstyle{\sum_{\eta\in\O_0}}\rho(\eta)Q_{0,A}f(\eta) \log Q_{0,A}f(\eta) \,.
	\end{align*}
Appealing to	Lemma \ref{submarlemma} and the martingale property of $\r_t$, 
\begin{align*}
		&\textstyle{\sum_{A\subset V}}\alpha_A\,\rho\left(\ent_A f\right)
		\\&\qquad \geq \textstyle{\sum_{A\subset V}}\alpha_A\,\bbE \left[\textstyle{\sum_{\eta\in\O_0}}\rho_1(\eta)f(\eta)\log f(\eta) - \textstyle{\sum_{\eta\in\O_0}}\rho_1(\eta)Q_{1,A} f(\eta) \log Q_{1,A} f(\eta)\right] \\
		& \qquad= \textstyle{\sum_{A\subset V}}\alpha_A\,\bbE\left[\ent_{\rho_1}f - \ent_{\rho_1}Q_{1,A}f\right] \,. %\\
	\end{align*}
One more application of the decomposition \eqref{eq:deco}, this time for the measure $\r_1$, shows that 
the last line in the above expression equals
\[
\textstyle{\sum_{A\subset V}}\alpha_A\,\bbE \left[\rho_1(\ent_{\rho_{1,A}}f)\right].
\] 
Since the probability measure $\r_1$ is a product over blocks $\eta^i$, $i\in[n]$, we may apply the Shearer inequality to obtain, almost surely	
	\begin{align*}
		\sum_{A \subset V}\alpha_A \,\rho_1\!\(\ent_{\rho_{1,A}}f\) \geq \gamma(\alpha)\,\ent_{\rho_1}f.
	\end{align*}
	Thanks to Lemma \ref{pottscovlemma} we have $\l(\cov[\cT_v\rho_t]) \leq \kappa(t)$ for all $t \in [0,1]$, $v\in\bbR^{qn}$,  where
	\begin{align*}
		\kappa(t):= \min\left\{ \frac{1}{2(1-(1-\delta)(1-t))}, \frac{n}{2} \right\}.
	\end{align*}
	Thus, the estimate in  Lemma \ref{prop1} holds with
	\begin{align*}
		a(t) = \min\left\{ \frac{1-\delta}{1-(1-\delta)(1-t)}, (1-\delta)n \right\}.
	\end{align*}
	If $\delta \in [\frac{1}{n},1]$, then we have $\frac{1-\delta}{1-(1-\delta)(1-t)} \le \frac{1-\delta}{\delta} \le (1-\delta)n$. In this case, we obtain that $a(t) = \frac{1-\delta}{1-(1-\delta)(1-t)}$ 
	and $\int_0^1a(t) dt = -\log \delta$.
	Meanwhile, if $\delta \in [0,\frac{1}{n})$, then 
	\begin{align*}
		a(t) = 
		\begin{cases}
			(1-\delta)n, & \text{if } 0 \le t \le \frac{\frac{1}{n}-\delta}{1-\delta}; \\
			\frac{1-\delta}{1-(1-\delta)(1-t)}, & \text{if } \frac{\frac{1}{n}-\delta}{1-\delta} < t \le 1. \\
		\end{cases}
	\end{align*}
	We thus obtain in this case that
	\begin{align*}
		\int_0^1a(t) dt &= \int_0^{\frac{\frac{1}{n}-\delta}{1-\delta}} (1-\delta)n dt + \int_{\frac{\frac{1}{n}-\delta}{1-\delta}}^1 \frac{1-\delta}{1-(1-\delta)(1-t)} dt \\
		&= 1-\delta n + \log n.
	\end{align*}
	Therefore, we deduce that
	\begin{align}\label{brownin1}
		\bbE\left[\ent_{\rho_1} f\right] \geq c_\delta \ent_\rho(f),
	\end{align}
	where
	\begin{align*}
		c_\delta= 
		\begin{cases}
			\frac{1}{ne^{1-\delta n}}, & \text{if } \delta\in [0,\frac{1}{n}); \\
			\delta, & \text{if } \delta\in [\frac{1}{n},1]. \\
		\end{cases}
	\end{align*}
	The desired estimate \eqref{sheaineq20} then follows.  
	This ends the proof of Theorem \ref{mainthm2}.
\end{proof}

%\begin{proof}[Proof of Corollary \ref{cor2-Ising}]
	The proof of Corollary \ref{cor2-Ising} is the same as for Theorem \ref{mainthm2}, with the only difference that, instead of using Lemma \ref{pottscovlemma}, we apply the better covariance estimate
	\begin{align}\label{eq:better-cov-est}
		\l\(\cov[\mathcal{T}_v\rho_t]\) \leq \min \left\{ \frac1{2(1-(1-\delta)(1-t))}, c_1 \sqrt{n} \right\},
	\end{align}
	where $c_1>0$ is an absolute constant.
        Namely, we replace the trivial upper bound of $n/2$ in \eqref{eq:covrov} by a tighter bound of $O(\sqrt{n})$ in \eqref{eq:better-cov-est}.
        Everything else of the proof remains exactly the same (simply replace $n$ with $2c_1\sqrt n$ in the above argument).
        We omit the details to avoid duplications.
	
    The $O(\sqrt{n})$ upper bound in \eqref{eq:better-cov-est} was established in \cite[Lemma 3.27]{CCYZ25} for general rank-one Ising models, but it can be readily obtained from even earlier works \cite{LLP10,DLP09} on the Curie-Weiss model.
	Crucially, it claims that for the critical Curie-Weiss model (i.e., $\beta = 1$) the maximum eigenvalue of its covariance matrix scales as $O(\sqrt{n})$, significantly better than the trivial $O(n)$ bound.
%\end{proof}
	%\newpage

	\subsection{Proof of Theorem \ref{th:SpInd} and Theorem \ref{th:critical-Ising}}
	We repeat the argument in the proof of Theorem \ref{mainthm2}, but this time we use a different bound for the function $a(s)$ appearing in Lemma \ref{prop1}. That is, rather than using Lemma \ref{pottscovlemma}, we use the following connection to spectral independence. Here we adopt the definition from \cite{CLV21,BCCPSV22}, which in turn is a generalized version of the original spectral independence definition from \cite{ALO}. Let $\cP_+(\O_0)$ denote the set of $\r\in\cP(\O_0)$ such that $\r(\eta^i_a = 1)>0$ for all $i\in[n], a\in[q]$. 
		\begin{definition}\label{def:SpInd}
	The \emph{influence matrix} of $\r\in\cP_+(\O_0)$ is the $qn\times  qn$ matrix $J[\r]$ defined by $J[\r](i,a;i,b) = 0$ for all $i\in[n], a,b\in[q]$, and
\begin{align}\label{def:J}
J[\r](i,a;j,b) = \r(\eta^j_b=1 \mid \eta^i_a = 1) - \r(\eta^j_b = 1) \quad \text{~for~} i \neq j.
\end{align}
For $K\ge0$, the distribution $\r$ is said to be $K$-spectrally independent if $\l(J[\r])$, the maximum eigenvalue of $J[\r]$, satisfies $\l(J[\r])\le K$.   
	\end{definition}
	We refer to \cite{CLV21,BCCPSV22} for various examples of spectrally independent spin systems, and for an illustration of the techniques, such as correlation decay and coupling estimates, that can be used to obtain bounds on $ \l(J[\r])$. For the present work, as in \cite{cheneldan}, all we need is the following simple relation between $J[\r]$ and $\cov[\r]$. 
	\begin{lemma}\label{lem:covJ}
For any $\r\in\cP_+(\O_0)$, the eigenvalues of $J[\r]$ are real and \[\l(\cov[\r])\le\l(J[\r])+ \frac12.\]  
	\end{lemma}
	\begin{proof} With the notation $p^i_a:=\r(\eta^i_a=1)$, we have 
	%Note that if $i\neq j$ then 
	$J[\r](i,a;j,b) =\cov[\r](i,a;j,b)/p^i_a$ for all $i\neq j$, and %while if $i=j$ then 
	\[
	J[\r](i,a;i,b) = 0 = \frac{\cov[\r](i,a;i,b)}{ p^i_a} - \IND_{a=b} + p^i_b\,.
	\]
	Let $D$ be the diagonal matrix with $D(i,a;j,b)=p^i_a\IND_{i=j, a=b}$. Then 
	\[
	J[\r] = D^{-1}\cov[\r]  - \IND + W,
	\] 
	where $W(i,a;j,b) = \IND_{i=j}p^i_b$, for all $a,b\in[q]$, $i,j\in[n]$.
	It follows that $J[\r]$ and
	\[J'[\r]:=D^{1/2}J[\r]D^{-1/2} = D^{-1/2}\cov[\r] D^{-1/2}- \IND + W'\] have the same eigenvalues, where $ W'(i,a;j,b) = \IND_{i=j}({ p^i_a p^i_b})^{1/2}$. Since $J'[\r]$ is 
	sym\-metric, the spectrum of $J[\r]$ is real. Moreover, $\cov[\r] = D^{1/2}J'[\r] D^{1/2}+ D - \bar W$, where $\bar W(i,a;j,b) = \IND_{i=j} p^i_a p^i_b$. Using Lemma \ref{lem:cov}, the largest eigenvalue of the block diagonal matrix $D-\bar W$ can be estimated by $\frac12$. Moreover, since $p^i_a\in[0,1]$,  one has		 
	\begin{align}\label{eq:covJ1}
\scalar{u}{D^{1/2}J'[\r] D^{1/2}u}\le \l(J'[\r])
\scalar{D^{1/2}u}{D^{1/2}u} \le  \l(J'[\r])\scalar{u}{u}\,,
\end{align}
for any $u\in\bbR^{qn}$.
Since $\l(J'[\r])=\l(J[\r])$, it follows that $\l(\cov[\r])\le\l(J[\r])+ \frac12$.
	\end{proof}

\bigskip
\begin{proof}[Proof of Theorem \ref{th:SpInd}]
%Proceeding as in the proof of Theorem \ref{mainthm2}, and using 
From \cite{CLV21} we know that if $\b$ is in the tree uniqueness regime with gap $\d\in(0,1)$, then it is $K$-spectrally independent with $\l(J[\cT_v\r])\le K = O(1/\d)$, uniformly in $v\in\bbR^{qn}$, where $\r$ is given in \eqref{isingmatr}.  
If $a(t)$ is the function from Lemma \ref{prop1}, we claim that
%$a(t)\le 2\l(\G)(K+\frac12)$, 
$a(t) = O(\frac{\l(\G)}{\delta+t})$,
for all $t\in[0,1]$. 
To see this, notice that the expression \eqref{stoclocmeasure} characterizes an Ising model with the effective inverse temperature given by $(1-t) \beta$. We then deduce that
\begin{align*}
    \tanh((1-t) |\beta|)
    &\overset{\text{(i)}}{\le} \tanh (|\beta|) - \frac{t|\beta|}{2} 
    \overset{\text{(ii)}}{\le} \tanh (|\beta|) \left( 1- \frac{t}{2} \right) \\
    &\overset{\text{(iii)}}{\le} \frac{(1-\delta) (1- \frac{t}{2} )}{\Delta-1} 
    \overset{\text{(iv)}}{\le} \frac{1 - \frac{\delta+t}{2}}{\Delta-1},
\end{align*}
where (i) follows from $\tanh(|\beta| - x) \le \tanh (|\beta|) - x/2$ for all $|\beta| \le 0.6$ and $0 \le x \le |\beta|$ by convexity, (ii) follows from $\tanh(|\beta|) \le |\beta|$, (iii) follows from our assumption \eqref{eq:tU}, and finally (iv) holds since $t\le 1$.
Therefore, such an Ising model is in the tree uniqueness regime with gap $(\delta+t)/2$ and
by Lemma \ref{lem:covJ} we obtain $a(t) = O(\frac{\l(\G)}{\delta+t})$.

To estimate $\l(\G)$, note that by adding a diagonal term we can take $\G(i,a;j,b):= |\beta|\Delta \IND_{i=j}\IND_{a=b}+\beta A(G)_{i,j}\IND_{a=b}$, which makes it a positive semidefinite matrix with maximum eigenvalue bounded by $2|\b|\D$. Therefore we may estimate $\l(\G)\le2|\b|\D$. The condition \eqref{eq:tU} implies that $   |\b|\D$ is bounded by a universal constant, and therefore 
%$a(t)\le C_0/\d$ 
$a(t)\le \frac{c_0}{\d+t}$ 
for all $t\in[0,1]$, for some universal constant $c_0$. 

Repeating the argument in Section \ref{sec:proofsec} we obtain
\begin{align}\label{sheaineq2001}
%		\sum_{A \subset V}\alpha_A\,\r\left(\ent_Af\right)\ge \gamma(\alpha) \,e^{-C_0/\d}\,\ent f ,
	\sum_{A \subset V}\alpha_A\,\r\left(\ent_Af\right)
	\ge \gamma(\alpha) \exp\left( - \int_0^1 \frac{c_0}{\delta+t} dt \right) \ent f 
	\ge \gamma(\alpha) (c_1 \delta^{c_2}) \ent f ,
\end{align}
for all functions $f:\O_0\mapsto \bbR_+$, where $c_1,c_2>0$ are universal constants.  
This ends the proof of the strong approximate Shearer inequality. The mixing time estimates follow from \eqref{fact_contr_mix}. 
\end{proof}
%
% the following result. 
%	\begin{theorem}\label{th:SpIndPotts}
%Consider the Potts model \eqref{For the ferromagnetic Ising model on a graph $G$ with maximum degree $\D\ge 3$, with $\b$ in the tree uniqueness regime with gap $\d\in(0,1)$, 
%the strong approximate Shearer inequality holds with constant $C_\d=e^{O(1/\d)}$. In particular,  for any $\a\in\cW$, the $\a$-weighted block dynamics has mixing time $C_\d\, \g(\a)^{-1}\times O(\log n)$. \end{theorem}
%
%In this case $\mu$ is $O(1/\d)$-spectrally indep. Therefore, using $\l(\cov(\cT_v\r))\le 1 + \l(J)$... one has $C_\a\le \exp O(1/\d)$.....
%
%	\newpage

\begin{remark}\label{rem:SpIndPotts}
We remark that, using the criterion in \cite[Theorem 4.13]{BCCPSV22} for spectral independence of ferromagnetic Potts models, a result analogous to Theorem \ref{th:SpInd} can be obtained for the  $q$-Potts model on a graph $G$ with maximum degree $\D$ provided $\b\in[0,\b_*)$ where $\b_*=\max\{\b_0,\b_1,\b_2\}$, where $\b_0= 2/\D$, $\b_1= (1/\D)\log((q-1)/\D)$, and $\b_2=\frac1{\D-1}\log(q)$. However, the dependence on $\D$ of the corresponding spectral independence constant $K$ is not optimal in this case; see \cite{BCCPSV22} for a more detailed discussion.  
\end{remark}

\begin{proof}[Proof of Theorem \ref{th:critical-Ising}]
	We proceed similarly as the proof of Theorem \ref{th:SpInd}.
	The upper bound for $a(t)$ is given by 
	\begin{align*}
		\int_0^1 a(t) dt = O(1) + \frac{\Delta}{\Delta-2} \log n,
	\end{align*}
	see \cite{CCYZ25}, Proof of Theorem 3.23.
	Thus, following the argument in Section \ref{sec:proofsec} we deduce that for all functions $f:\O_0\mapsto \bbR_+$,
%	\begin{align*}
%		\exp\left( \int_0^1 a(t) dt \right) = \exp\left( O(1) + \frac{\Delta}{\Delta-2} \log n \right) = O\left( n^{1+\frac{2}{\Delta-2}} \right).
%	\end{align*}
	\begin{align}
		\sum_{A \subset V}\alpha_A\,\r\left(\ent_Af\right)
		\ge \gamma(\alpha) \exp\left( - \int_0^1 a(t) dt \right) \ent f
		\ge \gamma(\alpha) \Omega\left( n^{-1-\frac{2}{\Delta-2}} \right) \ent f.
	\end{align}
	This establishes the strong approximate Shearer inequality with $C_n = O\left( n^{1+\frac{2}{\Delta-2}} \right)$, and the mixing time bound follows immediately.
\end{proof}

	\section{Proof of Theorem \ref{mainthm3} and  Theorem \ref{mainthm4}}
The proofs follow the same broad strategy as in Theorem \ref{mainthm2}, namely we first reduce to binary spins and then apply a stochastic localization process in order to obtain a new measure without interaction and with suitable random external fields. However, one has to pay attention to the fact that while the interaction between components is eventually deleted, the interaction within components may still be strong due to the possibly nontrivial nature of the original distributions $\mu_i$. The latter will be controlled using the constant $R$ defined in \eqref{def:R_i}. 
 
 As in Section \ref{sec:reduction}
 we start with a change of variables which maps the spin system on $\O=[q_1]^{L_1}\times\cdots\times[q_n]^{L_n}$ to a spin system on $\{0,1\}^{M}$, where $M:=q_1L_1 + \cdots q_nL_n$.
We define the new single spin spaces $\hat{\Om}_{0,i} = \left\{x\in\{0,1\}^{q_i}:\, \sum_{a=1}^{q_i}x_{a} =1\right\}$. In this way, $[q_i]$ is in one to one correspondence with $\hat{\Om}_{0,i}$,
and $\O$ is in one to one correspondence with $\Om_0:=\times_{i=1}^n\hat{\Om}_{0,i}^{L_i}$. Thus, 
a spin configuration $\si=\{\si_i(\ell), i\in[n],\ell\in [L_i]\}$ is associated to the new variables $\eta=(\eta^{i,\ell}_a)$, where $\eta^{i,\ell}_a:=\IND_{\si_i(\ell)=a}$, for all $i\in[n], \ell\in [L_i]$, and $a\in[q_i]$. 
The distribution $\mu_i$ over $\O_i=[q_i]^{L_i}$ is equivalent to the distribution $\r_i$ over $\hat\O_{0,i}$. We also write $\eta^i$, $i\in[n]$ for the ``block'' variable $\eta^i=(\eta^{i,\ell}_a,\,\ell\in[L_i],a\in[q_i])$. 
For $x,y\in\bbR^{M}$ we use the notation % $x=(x^{i,\ell}_a, i\in[n],\ell\in[L_i], a\in[q_i])$, %a y=(y^{i,\ell}_a, i\in[n], a\in[q])$,  and 
\[
\scalar{x}{y} = \textstyle{\sum_{i=1}^n\sum_{\ell=1}^{L_i}\sum_{a=1}^{q_i}}x^{i,\ell}_ay^{i,\ell}_a.
\]
The probability measure $\mu$ from \eqref{spinsys}, in the new variables $\eta\in\O_0$ becomes the probability $\r\in \cP(\O_0)$ given by
\begin{align}\label{isingmatrgen}
	\rho(\eta) \propto \exp\left(\scalar{\eta}{\G\eta}+\scalar{h}{\eta}\right)\prod_{i=1}^n\r_i(\eta^i)\,,
\end{align}
where $\G((i,\ell),a;(j,k),b) = \phi_{i,\ell;j,k}(a,b)$ is the matrix defined in \eqref{def:Gammagen},  and $h\in\bbR^{M}$, defined by $h^{i,\ell}_a := \psi_{i,\ell}(a)$, are the new external fields. %Here $\eta^i$ stands for the block variables $\eta^i = (\eta^{i,\ell}_a,\,\ell\in[L_i],a\in q_i)$. 
We note that the constraint $\ind_{\{\eta \in \Om_0\}}$ appearing in \eqref{isingmatr} here is implicit in the term $\prod_{i=1}^n\r_i(\eta^i)$ since each $\r_i$ is a probability on $\O_{0,i}$, and thus $\r(\eta)$ gives zero mass to $\eta\notin\O_0$.  Note also that if $\G=0$ then this measure is a product over the block variables $\eta^i$. 

	Next, we run the stochastic localization process on the new variables $\eta\in\O_0$, with the same construction given in Section \ref{sec:stloc}. 
When %apply it with 
$\r$ is given by the expression \eqref{isingmatrgen}, % for the initial distribution $\r$, 
we have that $\r_t\in\cP(\O_0)$ satisfies
\begin{align}\label{stoclocmeasuregen}
	\rho_t(\eta) \propto \exp\left((1-t)\scalar{\eta}{\Gamma\eta} + \scalar{y_t+h}{\eta}\right)\prod_{i=1}^n\r_i(\eta^i)\,, \ \ \ \ \ \ t \geq 0,
\end{align}
where $y_t := (2\G)^{\frac12}B_t + 2\Gamma\int_0^tm(F_s)ds$. Here $B_t$ is a Brownian motion on $\bbR^M$ and $m(F_s)$ is defined as in \eqref{eq:mt}.
Note that when $t=1$, we have a random product measure over the blocks $\eta^i$. The only difference with respect to   \eqref{prodmeasure} is that now 
\begin{align}\label{prodmeasuregen}
	\rho_1(\eta)=\prod_{i=1}^n\rho_{1,i}(\eta^i)\,,\qquad  \rho_{1,i}(\eta^i)\propto \exp\left(\scalar{y_1^i+h^i}{\eta^i}\right)\r_i(\eta^i)\,.
%	\rho_{1,i}(\eta^i)\propto \exp\left(\scalar{y_1^i+h^i}{\eta^i}\right)\IND_{\{\eta^i \in \hat\Om_0\}}\,,
\end{align}
%$\rho_{1,i}(\eta^i)\propto \exp\left(\scalar{y_1^i+h^i}{\eta^i}\right)\r_i(\eta^i)$ for each $i\in[n]$.
Moreover the bound in Lemma \ref{prop1} holds as it is. The next step is to estimate the covariance, as in Lemma \ref{pottscovlemma}.
Here we need to replace the uniform distribution over $\O_0$ with the product  $\nu:=\otimes \r_i$. This makes the analysis more involved, and it is the reason for the appearance of the constant $R$ defined in \eqref{def:R}. However, with this notation we have again \eqref{eq:decrho}
 and \eqref{eq:covtot}. 
Lemma \ref{pottscovlemma} now takes the following form.

\begin{lemma}\label{pottscovlemmagen}
	Consider the spin system $\rho\in\cP(\O_0)$ given by \eqref{isingmatrgen}, and assume that $\G$ is positive definite with $R\l(\G) <1$. Then for any $v \in \bbR^{M}$, $t\in[0,1]$, one has
	\begin{align}\label{eq:covrovgen}
\l\(\cov[\mathcal{T}_v\rho_t]\)\leq \frac{R}{2}\(\frac{1}{1-(1-t)R\l(\G))}\)\,.
	\end{align}
\end{lemma}
\begin{proof}
As in the proof of Lemma \ref{pottscovlemma} it suffices to 
 show the inequality \eqref{eq:covrovgen} with $t=0$ and with $\mathcal{T}_v\r_t$ replaced by 
 \begin{align}\label{eq:covrov11gen}	
 \rho_v(\eta) \propto\exp\left(\scalar{\eta}{\G\eta})+\scalar{v}{\eta}\right)\prod_{i=1}^n\r_i(\eta^i)\,.
\end{align}
We have the decomposition \eqref{eq:covtot} where now $\nu=\otimes \r_i$. We start with the analysis of the term $\cov [\mathcal{T}_{x}\nu]$, $x\in\bbR^M$. Since $\nu$ is a product over blocks $\eta^i$ it suffices to bound $\l(\cov[\cT_z\r_i])$ for any $z\in\bbR^{q_iL_i}$, for all $i\in[n]$.
 To this end, recalling the definition of the constant $R$ in \eqref{def:R}, we have 	$\l(\cov[\cT_z\r_i]) \le R/2$, for all $i\in[n]$, and $z\in\bbR^{q_iL_i}$. Thus, the bound \eqref{eq:covbo1} now takes the form
   \begin{align}\label{Rbound}
	\l(\cov[\cT_x\nu]) \le R/2\,,\qquad  \,x\in\bbR^{M}\,.
\end{align}
Therefore the last term in \eqref{eq:covtot} has maximum eigenvalue bounded by $R/2$. 
	 It remains to estimate the matrix $\cC$. Adapting  the notation introduced before, for any $u \in\bbR^{M}$ we write
	\begin{align}\label{eq:covygen}
		\scalar{u}{\cC u} &= \textstyle{\sum_{(i,\ell)}\sum_{(j,k)}\sum_{a,b}}u^{i,\ell}_au^{j,k}_b\cC((i,\ell),a;(j,k),b)=\var\left(g_u(X)\right)\,, 
			\end{align}
where $X$ is the $\bbR^{M}$-valued random variable with density $\hat f$ %$\var$ is the variance functional, 
and %we write 
$g_u(y):=\sum_{(i,\ell),a} u^{i,\ell}_{a}  w^{i,\ell}_a(y)$, with $w^{i,\ell}_a(y):=\sum_{\eta\in\O_0}\mathcal{T}_{y + v}\nu(\eta)\eta^{i,\ell}_a$, $i\in[n]$, $\ell\in[L_i]$, $a\in[q_i]$.
We may use again the Brascamp-Lieb estimate \eqref{eq:BLineq}.
Thus, we need to control
the norm $\|\grad g_u(y)\|_2^2$,
	where $\grad g_u(y)\in\bbR^{M}$ denotes the vector $[\grad g_u(y)]^{i,\ell}_a = \frac{\partial}{\partial y^{i,\ell}_a}g_u(y)$. In this case,
	\[
	[\grad g_u(y)]^{i,\ell}_a = \sum_{b=1}^{q_i} \sum_{k=1}^{L_i}u^{i,k}_{b}  \frac{\partial w^{i,k}_b(y)}{\partial y^{i,\ell}_a}=\sum_{b=1}^{q_i} \sum_{k=1}^{L_i}u^{i,k}_{b}  \cC^{i}(\ell,a;k,b)\,, %=c^2 \,(C^iu^i)_a\,.\
	\] 
where, for fixed $i\in[n]$, $\cC^i := (\cC^{i}(\ell,a;k,b),\,\ell,k\in[L_i],a,b\in[q_i])$ is the $q_iL_i\times q_iL_i$ matrix given 
by
\[
\cC^{i}(\ell,a;k,b) = \mathcal{T}_{y + v}\nu(\eta^{i,\ell}_a\eta^{i,k}_b) -\mathcal{T}_{y + v}\nu(\eta^{i,\ell}_a)\mathcal{T}_{y + v}\nu(\eta^{i,k}_b)\,.   
\] 
Therefore, by the definition %of the constant $R$ in 
\eqref{def:R_i}, the maximum eigenvalue of $\cC^i$ is bounded by $R/2$,  and we obtain, for all $y\in\bbR^{M}$,  
\[
	\|\grad g_u(y)\|_2^2=\sum_{(i,\ell),a} ([\grad g_u(y)]^{i,\ell}_a)^2 = \sum_{(i,\ell),a} (\cC^iu^i)^2(\ell,a) \le \frac{R^2}{4}\scalar{u}{u}\,.
	\] 
From \eqref{eq:BLineq} we conclude 	$\scalar{u}{\cC u} \leq \frac{R^2}{4h} \scalar{u}{u}$,
	where $h$ is a lower bound on the Hessian of $-\log \hat f$ as an $M\times M$ quadratic form. We have 
	\begin{align*}
-\log \hat f(y)= \frac{1}2\scalar{y}{(2\G)^{-1}y} - \log\left(Z\left(y + v\right)\right) + {\rm const.}
	\end{align*}
	Since ${\rm Hess}\log\left(Z\left(y + v\right)\right) = \cov[\mathcal{T}_{ y + v}\nu]$, by \eqref{Rbound} one has
	 ${\rm Hess}\log\left(Z\left(y + v\right)\right)\le  \frac{R}2{\rm Id}$. Moreover, ${\rm Hess}\scalar{y}{(2\G)^{-1}y}\ge \l(\G)^{-1}{\rm Id}$. Therefore we obtain the quadratic form bound
		\begin{align}\label{eq:BLineq31}
-{\rm Hess}\log \hat f(y) \ge \frac{1}2
\left(\frac{1}{\l(\G)} - R\right){\rm Id}.
	\end{align}
	%Summarizing, 
    Hence, we may take $h =  \frac{1}2
\left(\frac{1}{\l(\G)} - R\right)$ and then \eqref{eq:BLineq2} becomes $\scalar{u}{\cC u} \leq\frac{ R^2\l(\G)}{2(1-R\l(\G))}\scalar{u}{u}$. Together with the bound \eqref{Rbound} and \eqref{eq:covtot}, this concludes the proof of Lemma \ref{pottscovlemmagen}. 
\end{proof}
With these preparations we are ready to finish the proof of  Theorem \ref{mainthm3}.

 \subsection{Proof of Theorem \ref{mainthm3}}
We proceed as in the proof of Theorem \ref{mainthm2}. 
Appealing to 	Lemma \ref{submarlemma} and the martingale property of $\r_t$, 
\begin{align}\label{eq:b231}
		&\textstyle{\sum_{A\subset [n]}}\alpha_A\rho\left(\ent_A f\right)
\ge \bbE\sum_{A \subset [n]}\alpha_A \rho_1\left(\ent_{\rho_{1,A}}f\right).
	\end{align}
Since the probability measure $\r_1$ is a product over blocks $\eta^i$, $i\in[n]$, we may apply the Shearer inequality to these block variables to obtain almost surely
	\begin{align}\label{eq:b232}
		\sum_{A \subset [n]}\alpha_A \rho_1\left[\ent_{\rho_{1,A}}f\right] \geq \gamma(\alpha)\,\ent_{\rho_1}f.
	\end{align}
	Thanks to Lemma \ref{pottscovlemmagen} we have $\l(\cov[\cT_v\rho_t]) \leq \kappa(t)$ for all $t \in [0,1]$, $v\in\bbR^{M}$,  where $\kappa(t):= \frac{R}2\left(\frac{1}{1-(1-t  )R\l(\G)}\right)$. Thus, the estimate in  Lemma \ref{prop1} holds with $a(t)=\frac{R\l(\G)}{1-(1-t  )R\l(\G)}$. Since $\int_0^1a(t) dt = -\log(1-R\l(\G))$ we obtain 
	\begin{align}\label{brownin1gen}
		\bbE\left[\ent_{\rho_1} f\right] \geq (1-R\l(\G))\ent_\rho(f),
	\end{align}
	and the desired estimate \eqref{sheaineq2s} follows from \eqref{eq:b231}-\eqref{brownin1gen}.

 \subsection{Proof of Theorem \ref{mainthm4}}
 As already mentioned, even the case $\G=0$ of the statement requires some nontrivial analysis. We start with this case and then use the stochastic localization strategy to obtain the stated result. 
 \begin{lemma}\label{lem:ShearTenso}
	Suppose that $\mu=\otimes_{i=1}^n \mu_i\in\cP(\O)$ is a product of arbitrary components $\mu_i\in\cP(\O_i)$.
If $\mu_i$ satisfies the approximate Shearer inequality with constant $C_i$, for each $i\in[n]$, then $\mu$ satisfies the approximate Shearer inequality with constant $C=\max_iC_i$. 
\end{lemma}
\begin{proof}
	We apply a convexity argument similar to that of \cite[Lemma 3.2]{CP2021}, which establishes a closely related bound. Let $V_i$ denote the vertex set of the component $\mu_i$, each with size $|V_i|=L_i$, and write $V$ for the full vertex set of the system, with size $|V|=L_1+\cdots L_n$.	
	Iterating the decomposition \eqref{eq:deco} shows that 
	\begin{align}\label{eq:decost}
	\ent f=\textstyle{\sum_i}\mu\left[\ent_{V_{\le i}}\mu_{\leq i-1}f\right] ,
	\end{align} where $V_{\le i}:=V_1\cup\cdots\cup V_i$, and $\mu_{\leq i} := \otimes_{j=1}^i\mu_j$.	
	Let us fix the weights $\alpha = \{\alpha_A, A \subset V\}$ and an order $A_1,\dots,A_c$, $c=2^{|V|}$ of the subsets $A \subset V.$ Let us define $A_{i,j}:=V_i\cap A_j$ and $I_j := \{i:A_{i,j}\neq \emptyset\}$. Note that $A_j = \bigcup_{i \in I_j}A_{i,j}.$ Moreover, we define $A_{\leq i,j}:=\bigcup_{k \in I_j ,\, k \leq i}A_{k,j}.$ Arguing as in \eqref{eq:decost}, for each $j$ one has the decomposition 
	\begin{align}\label{eq:decop}
		\mu\left[\ent_{A_j}f\right] &= \sum_{i \in I_j}\mu\left[\ent_{A_{\leq i,j}}\mu_{A_{\leq i-1, j}}f\right] % \nonumber \\
		%&
		= \sum_{i \in I_j}\mu\left[\ent_{A_{i,j}}\mu_{A_{\leq i-1, j}}f\right] \,,
		%\\
%		&\geq \sum_{i \in I_j}\mu \left[\ent_{A_{i,j}}\mu_{\leq i-1}f\right],
	\end{align}
	where the last identity uses the product structure of $\mu$.
	
	The key observation is that, by convexity, for all $i\in I_j$, 
	\begin{align}\label{eq:decop2}
		\mu\left[\ent_{A_{i,j}}\mu_{A_{\leq i-1, j}}f\right] \ge \mu \left[\ent_{A_{i,j}}\mu_{\leq i-1}f\right]\,.
	\end{align}
The proof of \eqref{eq:decop2} uses the product structure of $\mu$, see the proof of \cite[Lemma 3.2]{CP2021} for the details.
By averaging with respect to the weights $\a$, and exchanging the summations, %one has
	\begin{align}
		\sum_{j}\alpha_{A_j}\mu\left[\ent_{A_j}f\right] \geq %\sum _j\sum_{i \in I_j}\alpha_{A_j}\mu\left[\ent_{A_{i,j}}\mu_{\leq i-1}f\right] \\
		 \sum_i\sum_{j :\, i \in I_j}\alpha_{A_j}\mu\left[\ent_{A_{i,j}}\mu_{\leq i-1}f\right] \label{eq2}
	\end{align}
	Now let $\a^i=(\a^i_U,\,U\subset V_i)$ denote the weights defined by $\a^i_{U}:=\sum_{j:\,i\in I_j}\IND_{U=A_{i,j}}\a_{A_j}$. 
	The approximate Shearer inequality for $\mu_i$, with constant $C_i$, applied to the function $\mu_{\leq i-1}f$, yields 
	\begin{align}
		&C_i\,\textstyle{\sum_{j :\, i \in I_j}}\alpha_{A_j}\mu_i\left[\ent_{A_{i,j}}\mu_{\leq i-1}f\right]
		\nonumber\\&\;\; =C_i\,\textstyle{\sum_{U\subset V_i}}\alpha^i_{U}\mu_i\left[\ent_{U}\mu_{\leq i-1}f\right] \label{eq3}
		\ge \g_i(\a^i)\ent_{V_i}\mu_{\leq i-1}f\,,
	\end{align} 
	where $\gamma_i(\alpha^i) := \min_{x \in V_i}\sum_{j :\, x \in A_j}\alpha_{A_j}$. Since $ \gamma_i(\alpha^i)\geq \gamma(\alpha)$, 
	averaging over $\mu$, and taking $C=\max_iC_i$, from \eqref{eq2}-\eqref{eq3} we obtain
	\begin{align}
	C	\,\textstyle{\sum_{j}}\alpha_{A_j}\mu\left[\ent_{A_j}f\right] \geq \g(\a)\sum_i\mu\left[\ent_{V_i}\mu_{\leq i-1}f\right] \,.
	%= \g(\a)\,\ent f\,.\label{eq4}
	\end{align}
Since $\mu$ is a product of $\mu_i$'s, we may replace $\mu\left[\ent_{V_i}\mu_{\leq i-1}f\right] $ with $\mu\left[\ent_{V_{\le i}}\mu_{\leq i-1}f\right]$ and using \eqref{eq:decost} concludes the proof.
%	Iterating the decomposition \eqref{eq:deco} shows that \[\textstyle{\sum_i}\mu\left[\ent_{V_{\le i}}\mu_{\leq i-1}f\right] =\ent f,\] where $V_{\le i}=V_1\cup\cdots\cup V_i$. Moreover, the product structure of $\mu$ implies that $\ent_{V_{\le i}}\mu_{\leq i-1}f=\ent_{V_i}\mu_{\leq i-1}f$ and therefore \eqref{eq4} and the arbitrariness of the weights $\a$ conclude the proof. 
\end{proof}

\bigskip

We are now able to conclude the proof of  Theorem \ref{mainthm4}. The estimate \eqref{brownin1gen} allows us to reduce the problem to an estimate on the entropy $\ent_{\r_1}f$, where $\r_1$ is the product over blocks $\eta^i$ given in \eqref{prodmeasuregen}. By the assumption that $\mu_i$ satisfies the strong approximate Shearer inequality with constant $C_i$, we know that, almost surely, each $\r_{1,i}$ must satisfy the approximate Shearer inequality with constant $C_i$. Thus, by Lemma \ref{lem:ShearTenso} we infer that, almost surely,  $\r_1$ satisfies
the approximate Shearer inequality with constant $C=\max_iC_i$. 
Therefore, if $\a=(\a_A, \,A\subset V)$ is an arbitrary choice of weights for the full system, %one has
 	\begin{align}\label{eq:b2321}
		C\,\textstyle{\sum_{A \subset V}}\alpha_A\, \rho_1\!\left(\ent_{\rho_{1,A}}f\right) \geq \gamma(\alpha)\,\ent_{\rho_1}f.
	\end{align}
Taking the expectation,  and using Lemma \ref{submarlemma} as in \eqref{eq:b232} and \eqref{brownin1gen} we conclude that 
\begin{align}\label{eq:b231a}
		C\,\textstyle{\sum_{A\subset V}}\alpha_A\rho\left(\ent_A f\right)
\ge \gamma(\alpha)\,\bbE\left[\ent_{\rho_1}f\right]\ge \g(\a)(1-R\l(\G))\ent_\rho(f).
	\end{align}
This concludes the proof of the first statement in Theorem \ref{mainthm4}. To prove the second statement, namely that we can replace $R$ with $\max_iC_i$, it suffices to show that $R\le \max_iC_i$. This is the content of the next lemma,
 which generalizes
Lemma \ref{lem:cov} and provides a useful relation between the approximate Shearer inequality and the constant $R$.

\begin{lemma}\label{lem:Rlemma}
		Suppose each component $\mu_i$ satisfies the strong approximate Shearer inequality with constant $C_i$. 	
	Then 
		\[ 
		R\le \max_iC_i.
		\] 
	\end{lemma}
	\begin{proof}
	Fix $i\in [n]$ and the self potentials $\psi$. It suffices to prove that $\l(\Sigma^\psi_i)\le C_i/2$. With the notation from \eqref{def:R_i},
			for any $u^i\in\bbR^{q_iL_i}$ one has
		\[
		\textstyle{\sum_{\ell,k}\sum_{a,b}}u^{i,\ell}_a\Sigma^\psi_i(\ell,a;k,b)u^{i,k}_b= \var_{i}(\scalar{u^i}{\o^i})\,,
		% = \frac12\textstyle{\sum_{a,b}}m_am_b(u_a-u_b)^2\,,
		\]
		where $\o^i\in\bbR^{q_iL_i}$ is the random vector $\o^i=(\o^{i,\ell}_a, \ell\in[L_i],a\in[q_i])$ with $\o^{i,\ell}_a = \IND_{\si_i(\ell)=a}$, $\var_{i}$ denotes the variance with respect to ${\mu^{\psi}_i}$ and we write 
		$\scalar{u^i}{\o^i}=\sum_{(\ell,a)}u^{i,\ell}_a\o^{i,\ell}_a$.  By linearizing the Shearer inequality for $\mu^{\psi}_i$ one obtains the variance inequality
		\begin{align}\label{eq:SheaVar}
C_i		\textstyle{\sum_{A}}\alpha_{A}\mu^{\psi}_i\left[\var_{i,A}f\right] \geq \g(\a)\var_{i}f \,,
	\end{align}
	where the sum runs over subsets $A\subset [L_i]$, $\a$ are arbitrary weights, $\var_{i,A}f$ denotes the variance with respect to the conditional distribution obtained from $\mu^{\psi}_i$ by freezing all variables $\si_i(\ell)$, $\ell\notin A$, and $\g(\a)=\min_\ell\sum_{A: \,\ell\in A}\a_A$; see e.g.\ \cite{PCnotes} for the standard linearization argument. Specializing to $\a_A = L_i^{-1}\IND_{|A|=1}$ one has the Efron-Stein inequality
		\begin{align}\label{eq:SheaVar2}
C_i		\,\textstyle{\sum_{\ell}}\,\mu^{\psi}_i\left[\var_{i,\ell}f\right] \geq \var_{i}f \,,
	\end{align}
where $\var_{i,\ell}f$ denotes the variance with respect to the conditional distribution obtained from $\mu^{\psi}_i$ by freezing all variables $\si_i(k)$, $k\neq \ell$. Taking $f= \scalar{u^i}{\o^i}$ we see that for each fixed pair $i,\ell$, 
	\[
	\var_{i,\ell}f = \var_{i,\ell}(\textstyle{\sum_{a}}u^{i,\ell}_a\o^{i,\ell}_a)=	%\textstyle{\sum_{a,b}}u_aC_{a,b}u_b = \var(\scalar{u}{\o}) = 
	\frac12\textstyle{\sum_{a,b\in [q_i]}}m_am_b(u^{i,\ell}_a-u^{i,\ell}_b)^2\,,
		\]
	where $m_a$ is the probability of the event $\si_i(\ell)=a$ conditionally on all variables $\si_i(k)$, $k\neq \ell$. Thus we may use the estimate in the proof of Lemma \ref{lem:cov} to conclude that $\var_{i,\ell}f \le \frac12\textstyle{\sum_{a}}(u^{i,\ell}_a)^2$.
	In conclusion,
	\[
	 \var_{i}(\scalar{u^i}{\o^i})\le \tfrac{1}2\,C_i\, \textstyle{\sum_{(\ell,a)}}(u^{i,\ell}_a)^2,
	 \]
	 which implies that the maximum eigenvalue of $\Sigma^\psi_i$ is bounded by $C_i/2$. 
		\end{proof}

Another class of distributions for which a bound on $R$ can be obtained is the product of Bernoulli distributions conditioned on their sum. 
\begin{lemma}\label{lem:Rlemma2}
		Suppose $q=2$, and, for each $i\in[n]$, let $\mu_i$ be a product measure conditioned on a fixed value of the sum, namely $\mu_i=\otimes_\ell \mu_{i,\ell}(\cdot\,|\,\sum_{\ell}\si_i(\ell)=k_i)$ for some $k_i\in\bbN$ and some arbitrary Bernoulli probability measures $\mu_{i,\ell}$. 
%		Then, or all choices of $\psi$,  	
%		\[ 
%		\mu_i^\psi(\IND_{\{\si_i(\ell)=1\}}\IND_{\{\si_i(k)=1\}})\le \mu_i^\psi(\IND_{\{\si_i(\ell)=1\}}) \mu_i^\psi(\IND_{\{\si_i(k)=1\}}) 
%	\,,\qquad  \ell\neq k.	\] 
		Then $R\le 2$, and the bound is optimal. 
	\end{lemma}
	\begin{proof}
We fix some $i\in[n]$ and $k_i\in\bbN$.	As in the proof of Lemma \ref{lem:Rlemma}, %With the notation from \eqref{def:R_i},
			for any $u^i\in\bbR^{q_iL_i}$ one has
		\[
		\textstyle{\sum_{\ell,k}\sum_{a,b}}u^{i,\ell}_a\Sigma^\psi_i(\ell,a;k,b)u^{i,k}_b= \var_{i}(\scalar{u^i}{\o^i})\,,
		% = \frac12\textstyle{\sum_{a,b}}m_am_b(u_a-u_b)^2\,,
		\]
		where $\o^i\in\bbR^{q_iL_i}$ is the random variable $\o=(\o^{i,\ell}_a, \ell\in[L_i],a\in[q_i])$ with $\o^{i,\ell}_a = \IND_{\si_i(\ell)=a}$. 	Since $q=2$, 		
		\[\scalar{u^i}{\o^i}=\textstyle{\sum_{\ell}}\left(u^{i,\ell}_1\o^{i,\ell}_1 + u^{i,\ell}_2(1-\o^{i,\ell}_1)\right)=
		\textstyle{\sum_{\ell}}v_\ell X_\ell + {\rm const.}\]
		where $v_\ell:=(u^{i,\ell}_1- u^{i,\ell}_2)$ and $X_\ell:=\o^{i,\ell}_1\in\{0,1\}$ is a Bernoulli random variable. 
		Under $\mu_i^\psi$, the $X_\ell$ are independent random Bernoulli variables conditioned on their sum.  
		Such random variables have the so-called strong Rayleigh property \cite{BBL2009}, and therefore they are negatively correlated, for all choices of the external fields $\psi$. Thus the matrix $\cC(\ell,k):=\mu_i^\psi(X_\ell X_k)-\mu_i^\psi(X_\ell)\mu_i^\psi(X_k)$ satisfies $\cC(\ell,k)\le 0$ for all $\ell\neq k$. Moreover, $\sum_k \cC(\ell,k)=0$ for all $\ell$, because of the constraint $\sum_{\ell}\si_i(\ell)=k_i$. Thus, \[
		\textstyle{\sum_{\ell:\,\ell\neq k}} |\cC(\ell,k)| = \cC(k,k) = m_k(1-m_k),
		\] where $m_k=\mu_i^\psi( X_k)$. In conclusion, we have
		\begin{align*}
		\var_{i}(\scalar{u^i}{\o^i}) &= \var_{i}\(\textstyle{\sum_{\ell}}v_\ell X_\ell\) = \tfrac12\textstyle{\sum_{\ell,k}}(v_\ell-v_k)^2|\cC(\ell,k)|\\
		&\le \textstyle{\sum_{\ell,k}}(v^2_\ell+v^2_k)\IND_{\ell\neq k}|\cC(\ell,k)|=2\sum_k v_k^2m_k(1-m_k)		\\
		&\le \tfrac12 \textstyle{\sum_k} v_k^2
\le \textstyle{\sum_{k}}((u^{i,\ell}_1)^2+ (u^{i,\ell}_2)^2)=\scalar{u^i}{u^i}\,.
		\end{align*}
This shows that $\l(\Sigma^\psi_i)\le 1$, and therefore $R\le 2$. 

To prove that the bound $R\le 2$ %in Lemma \ref{lem:Rlemma2} 
cannot be improved, take $L_i=2$, $q_i=2$, $k_i=3$, and the distribution $\mu_i$ which gives probability $1/2$ to 
$(\si_i(1)=1,\si_i(2)=2)$ and probability $1/2$ to 
$(\si_i(1)=2,\si_i(2)=1)$. Then $\Sigma_i$ is the $4\times 4$ matrix $\Sigma_i(\ell,a;k,b)$, $\ell,k\in[2]$, $a,b\in[2]$ given by 
$\Sigma_i=\frac14 J$ where $J$ has first and fourth row equal to $(1,-1,-1,1)$ and second and third row equal to $-(1,-1,-1,1)$. Then $\Sigma_i$ has maximal eigenvalue $1$ and thus $R=2$ in this case.  
	\end{proof}

\bibliographystyle{plain}

\bibliography{references}

\end{document}